\documentclass[11pt,twoside]{amsart}
\usepackage[latin1]{inputenc}
\usepackage[T1]{fontenc}
\usepackage{mathtools}
\usepackage{graphicx}
\usepackage{tikz}
\usetikzlibrary{chains}
\usepackage{tcolorbox}

\usepackage[latin1]{inputenc}
\usepackage{amsmath}
\usepackage{amsthm}
\usepackage{amssymb}

\usepackage[all]{xy}
\usepackage{hyperref}
\newtheorem{thm}{Theorem}[section]

\newtheorem{prop}[thm]{Proposition}

\newtheorem{lem}[thm]{Lemma}
\newtheorem{cor}[thm]{Corollary}

\numberwithin{equation}{section}

\theoremstyle{definition}
\newtheorem{definition}[thm]{Definition}
\newtheorem{remark}[thm]{Remark}

\newcommand{\End}{{\rm End}}

\newcommand{\cal}{\mathcal}

\newcommand{\ks}{{\cal S}}

\newcommand{\kz}{{\cal Z}}

\newcommand{\ZZ}{\mathbb{Z}}
\newcommand{\QQ}{\mathbb{Q}}
\newcommand{\RR}{\mathbb{R}}
\newcommand{\CC}{\mathbb{C}}

\newcommand{\PP}{\mathbb{P}}

\renewcommand{\to}{\xymatrix@1@=15pt{\ar[r]&}}
\renewcommand{\rightarrow}{\xymatrix@1@=15pt{\ar[r]&}}
\renewcommand{\leftarrow}{\xymatrix@1@=15pt{&\ar[l]}}
\renewcommand{\mapsto}{\xymatrix@1@=15pt{\ar@{|->}[r]&}}
\renewcommand{\twoheadrightarrow}{\xymatrix@1@=18pt{\ar@{->>}[r]&}}
\renewcommand{\hookrightarrow}{\xymatrix@1@=15pt{\ar@{^(->}[r]&}}
\newcommand{\hook}{\xymatrix@1@=15pt{\ar@{^(->}[r]&}}
\newcommand{\congpf}{\xymatrix@L=0.6ex@1@=15pt{\ar[r]^-\sim&}}
\renewcommand{\cong}{\simeq}

\usepackage{comment}
\excludecomment{comment}
\begin{document}

\title[]{Complex multiplication in twistor spaces}

\author[D.\ Huybrechts]{D.\ Huybrechts}

\address{Mathematisches Institut and Hausdorff Center for Mathematics,
Universit{\"a}t Bonn, Endenicher Allee 60, 53115 Bonn, Germany}
\email{huybrech@math.uni-bonn.de}

\begin{abstract} \noindent
Despite the transcendental nature of the twistor construction,
the algebraic fibres of the twistor space of a K3 surface
share certain arithmetic properties. We prove that 
for a polarized K3 surface with complex multiplication, all
algebraic fibres of its twistor space away from the equator have complex multiplication as well.
 \vspace{-2mm}
\end{abstract}

\maketitle

Let $S$ be a complex projective K3 surfaces with an ample class $\ell={\rm c}_1(L)\in H^2(S,\ZZ)$. Viewing
$\ell$ as a K\"ahler class and using the existence of a Ricci-flat K\"ahler form representing $\ell$,
one constructs the twistor space which consists of a non-algebraic complex threefold $\ks$ and 
a natural holomorphic
projection $\ks\to\PP^1$. All
fibres $\ks_t$ are K3 surfaces, but only countably many of them are again algebraic. However, the set of algebraic fibres
is dense and away from the equator of $\PP^1\cong S^2$ all algebraic fibres $\ks_t$
have the same Picard number $\rho(S)$.

The main result of this paper proves that despite the transcendental nature of the twistor construction, which relies
on  Yau's solution of the Calabi conjecture \cite{Yau}, the K3 surface $S$ passes on
certain arithmetic features to other algebraic fibres of the twistor family.
We will demonstrate this for K3 surfaces with complex multiplication (CM). 

A  K3 surfaces $S$ is said to have CM if the
endomorphism ring $K_S\coloneqq{\rm End}_{\rm Hdg}(T(S))$ of the Hodge structure provided by the rational transcendental lattice
$T(S)\coloneqq{\rm NS}(S)^\perp\otimes\QQ$ is a CM field with $\dim_{K_S}T(S)=1$, see Sections \ref{sec:EndoT}, and \ref{sec:CMK3}.

K3 surfaces with CM are  defined over number fields and they are
exactly those K3 surfaces that are defined over number fields and have algebraic periods \cite{Tretkoff}. 
Examples include all K3 surfaces of maximal Picard number $20$, for which
the following theorem is immediate. Note, however, that there exist infinitely many K3 surfaces with CM 
and arbitrarily fixed even Picard number $\leq 20$. \footnote{More precisely, for any CM field of degree $\leq20$ there exists a K3 surface with CM that realizes the prescribed CM field, see \cite[Prop.\ 3.3.8]{HuyK3} for references. It seems likely that the set of
polarized K3 surfaces with CM of fixed degree is dense in the moduli space of polarized K3 surfaces, but I am not aware of a reference for this.}

\begin{thm}\label{thm:main}
Let $(S,L)$ be a polarized, complex K3 surface with complex multiplication and assume $S'$ is an algebraic
fibre of the associated twistor space $\ks\to\PP^1$ over a point not contained in the equator $S^1\subset S^2$
(cf.\ Sections \ref{sec:ReviewTwistor}, \ref{sec:TwistorConstrT}, and \ref{sec:Dict}).
\begin{enumerate}
\item[{\rm (i)}] Then $S'$ has complex multiplication. 
\item[{\rm (ii)}] The maximal totally real subfields of the CM endomorphism fields $K_S$ and $K_{S'}$ coincide.
\end{enumerate}
\end{thm}

More geometrically, the endomorphism field $K_S={\rm End}_{\rm Hdg}(T(S))$ can be viewed,
 using Poincar\'e duality, as the non-trivial
part of the space of Hodge classes on $S\times S$, cf.\ Remark \ref{rem:EndHodgeS2}:
$$K_S\cong (T(S)\otimes T(S))^{2,2}\subset H^{2,2}(S\times S,\QQ)\cong\left({\rm NS}(S)\otimes {\rm NS}(S)\right)_\QQ\oplus 
(T(S)\otimes T(S))^{2,2}.$$
The maximal totally real subfield of $K_S$ is the subspace $(S^2T(S))^{2,2}$, which is a subfield of index at most two. Although non-trivial classes in this subspace do not
deform along the twistor family $\ks\to\PP^1$,  cf.\ Section \ref{sec:EndoasHodge}, the same classes recur  at any other
algebraic fibre $S'$ away from the equator. More precisely, there is a multiplicative isomorphism
$$(S^2T(S))^{2,2}\cong (S^2T(S'))^{2,2},$$
where the product on both sides corresponds to the composition of correspondences
given by classes on the squares of $S$ and $S'$.

\smallskip


\medskip

\noindent
{\bf Outline:} In Section \ref{sec:Twistor} one finds a quick reminder of the twistor construction. 
In the subsequent Sections \ref{sec:PeriodsHodge} and \ref{sec:HodgeTwistor},
we discuss abstract Hodge structures of K3 type and explain the twistor construction
in this setting. In particular, we show that the CM property is equivalent
to the equality $k_T=K_T$ of the period and the endomorphism field  (Proposition \ref{prop:CMKk}) and that
excessive  jumps of the Picard number only occur along the equator (Proposition \ref{prop:Picardjump}). The main
result in this part is Proposition \ref{prop:FinalHodge}, which is the Hodge theoretic version of Theorem \ref{thm:main}.
The final Corollary \ref{cor:CMfieldreconstruction} explains how to reconstruct the CM field $K_{S'}$ from its maximal
totally real subfield $K_{S'}^0=K_S^0$. In Section \ref{sec:PV} we discuss the notion of the period value of
a K3 surfaces defined over a number field in the abstract Hodge theoretic setting and compute the
period values of all CM twistor fibres.
Section \ref{sec:K3CM} translates the abstract  Hodge theory into geometric results and proves Theorem \ref{thm:main}. The section also contains a proof of the known fact that K3 surfaces with CM are defined
over number fields that does not use the Kuga--Satake construction (Proposition \ref{prop:CMimpliesQbar}), and a discussion of transcendental periods.

\medskip

\noindent
{\bf Acknowledgments:} I am grateful to F.\ Charles for a discussion related to \cite{HuyAnn} that also influenced this paper. Thanks also to J.\ de Jong and W.\ Sawin for a discussion
that helped to put straight my ideas about period values of twistor fibres. I am deeply indebted to two referees
for their careful reading of the paper and many interesting and helpful suggestions.

\section{Review of the twistor construction for K3 surfaces}\label{sec:Twistor}
For basic facts about hyperk\"ahler geometry and the twistor construction, the reader may consult
the survey \cite{GHJ,Hitchin} or the extensive \cite{Besse}. Here, we merely sketch the main features and
stress the analytic nature of the construction.

\subsection{Ricci flat metrics}\label{sec:ReviewTwistor}
Let $S$ be a K3 surface and think of it as a differentiable manifold $M$ endowed with a complex structure $I$. Then, according to \cite{Yau}, any K\"ahler class $\alpha\in H^{1,1}(S,\RR)$ is represented by a unique K\"ahler form $\omega$ such that the two real volume forms $\omega^2$ and $\sigma\bar\sigma$, where $0\ne \sigma\in H^{2,0}(S)$, differ only by a constant scalar. After normalizing $\sigma$ appropriately, we may
assume $\omega^2=\sigma\bar\sigma$.
In terms of the complex structure $I$ and the underlying Ricci-flat K\"ahler metric $g$, the K\"ahler
form $\omega$ can be written as $\omega=g(I(~),~)$.

As the holomorphic volume form $\sigma$ is actually a holomorphic symplectic structure,
the holonomy group of $g$ is ${\rm Sp}(1$). In particular, there are two other complex
structures $J$ and $K$ compatible with the metric $g$ and such that $I=J\cdot K=-K\cdot J$. In fact,
any linear combination $I_t\coloneqq aI+bJ+cK$ with $t=(a,b,c)\in S^2$ defines a complex
structure on $M$ compatible with the given K\"ahler metric $g$. We denote the corresponding
K\"ahler forms by $\omega_t=\omega_{I_t}=g(I_t(~),~)$, for example, $\omega=\omega_I=\omega_{(1,0,0)}$.
In fact, all the complex surfaces $(M,I_t)$ constructed in this way are K3 surfaces and, in particular, come with
a holomorphic volume form $\sigma_t\coloneqq\sigma_{I_t}$ for which we may assume $\omega_t^2=\sigma_t\bar\sigma_t$, see \cite[Sec.\ 3.F]{HKLR}.
Note also that $\sigma=\sigma_I=\omega_J+i\omega_K$ and 
that in fact all $\omega_t$ are contained in the $\RR$-linear span of $\omega_I,\omega_J$, and $\omega_K$.
Furthermore, the real and imaginary parts of $\sigma_t$  span the orthogonal complement of $\omega_t$
in $\langle\omega_I,\omega_J,\omega_K\rangle_\RR$:
$$\RR\cdot{\rm Re}(\sigma_t)\oplus^\perp\RR\cdot{\rm Im}(\sigma_t)\oplus^\perp \RR\cdot \omega_t=\langle\omega_I,\omega_J,\omega_K\rangle_\RR.$$

On the differentiable manifold $M\times S^2$ one defines the almost complex structure
${\mathbb I}$ at a point $(x,t)\in M\times S^2$
as $I_t\times I_{\PP^1}$, where $S^2$ is interpreted as the complex projective line $\PP^1$. A natural identification in this context will be explained below, cf.\ (\ref{eqn:PGRS}). It turns out that ${\mathbb I}$ 
is in fact integrable. The resulting complex manifold defined in this way 
together with the holomorphic projection to the second factor shall be denoted by
 $$\ks\coloneqq (M\times S^2,{\mathbb I})\to S^2\cong\PP^1.$$
The fibres $\ks_t$ of the projection are the K3 surfaces $(M,I_t)$.
We think of $(1,0,0)$, which corresponds to the original complex structure $I$,
as the north pole of $S^2$. Then the fibre over the south pole $(-1,0,0)$ is the K3 surface $(M,-I)$,
the complex conjugate of $S$. The \emph{equator} $\{(0,b,c)\mid b^2+c^2=1\}$ parametrizes
all complex structures $I_t$ for which $\omega$ is contained in the plane
$\langle{\rm Re}(\sigma_t),{\rm Im}(\sigma_t)\rangle_\RR$.

In Section \ref{sec:HodgeTwistor} we provide a description of the twistor construction in terms of the involved
Hodge structures, which in Section \ref{sec:K3CM} will then be translated back into results for $\ks\to \PP^1$.

\subsection{Twistor spaces associated with ample classes}
The twistor construction relies on the existence of the Ricci-flat metric $g$. The existence
of $g$ is guaranteed by \cite{Yau}, but it cannot be constructed in an explicit way and is thought of as
a transcendental structure. For example, if $S$ is embedded into $\PP^N$, so in particular $S$ is projective,
then the restriction of the Fubini--Study metric on $\PP^N$ to $S$ is never Ricci-flat \cite{Hullin}.
One may want to compare this result to \cite{Donaldson}, where it is shown that for appropriately chosen
projective embeddings $\varphi_n\colon
S\,\hookrightarrow \PP^{N_n}$ associated with the linear systems $|L^n|$, $n\to\infty$, the pull-backs
$(1/n)\varphi_n^*\omega_{\rm FS}$ of the Fubini--Study K\"ahler forms on $\PP^{N_n}$ approach
the Ricci-flat K\"ahler form representing the ample class ${\rm c}_1(L)$.

For an arbitrary K\"ahler class $\alpha\in H^{1,1}(S,\RR)$ there is very little one can
say about the various fibres $\ks_t$ of the associated twistor space. However, more
structure emerges when $\alpha$ is an ample class $\ell={\rm c}_1(L)$.

\begin{remark}
Note that the twistor space $\ks\to\PP^1$ associated with $S$ and a K\"ahler
class $\alpha=[\omega]$ can be viewed as the twistor space associated with
an arbitrary fibre $\ks_t$ endowed with the K\"ahler class $[\omega_t]$.
However, the property that $\alpha$ is an ample class $\ell$ is not preserved, i.e.\
$[\omega_t]$ will rarely be integral or rational again.
\end{remark}

\subsection{Hyperk\"ahler manifolds}\label{sec:HK} In the case of higher-dimensional hyperk\"ahler manifolds,
i.e.\ simply-connected, compact K\"ahler manifolds $X$ for which $H^0(X,\Omega_X^2)$ is spanned by an
everywhere non-degenerate form, the condition $\omega^2=\sigma\bar\sigma$, up to constant scaling,
has to be replaced by $\omega^{2n}=(\sigma\bar\sigma)^n$. Here, $2n$ is the complex dimension of $X$.
Again, the situation is controlled by the Hodge structure of weight two $H^2(X,\ZZ)$ which
is endowed with the Beauville--Bogomolov form. Theorem \ref{thm:main} remains valid in
higher dimensions, for the result is ultimately deduced from purely Hodge theoretic arguments
in Section \ref{sec:HodgeTwistor} and those apply to the transcendental part $T(S)\subset H^2(S,\QQ)$
of a projective K3 surface $S$ as well as to the transcendental part $T(X)\subset H^2(X,\QQ)$ of
a projective hyperk\"ahler manifold $X$.


\section{Hodge structures of K3 type: Periods and endomorphisms}\label{sec:PeriodsHodge}

In the following $T$ will denote a rational Hodge structure of \emph{K3 type}. Concretely, this means that $T$ is a $\QQ$-vector space
of finite dimension $r$ endowed with a symmetric bilinear form $(~.~)$ of signature $(2,r-2)$ or
$(3,r-3)$ and a decomposition
\begin{equation}\label{eqn:HS}T_\CC\coloneqq T\otimes_\QQ\CC=T^{2,0}\oplus T^{1,1}\oplus T^{0,2}\end{equation} such that
with respect to the $\CC$-linear extension of $(~.~)$ the following conditions are satisfied:
\smallskip

\noindent
(i)   The subspaces $T^{1,1}$ and $T^{2,0}\oplus T^{0,2}$ are orthogonal.
\smallskip

\noindent
(ii) $(~.~)$ is positive definite on $P_T\coloneqq(T^{2,0}\oplus T^{0,2})\cap T_\RR$ and $T^{2,0},T^{0,2}\subset T_\CC$ are isotropic.
\smallskip

\noindent
(iii) Complex conjugation on $T_\CC$ preserves $T^{1,1}$ and exchanges $T^{2,0}$ and $T^{0,2}$, i.e.\ $\overline{T^{2,0}}=T^{0,2}$.
\smallskip

\noindent
(iv) $\dim_\CC T^{2,0}=1$.

\medskip

A generator of $T^{2,0}$ will usually be called $\sigma$, i.e.\ $T^{2,0}=\CC\cdot\sigma$. Note that giving a decomposition
(\ref{eqn:HS}) satisfying (i)-(iv) is equivalent to giving $\sigma\in T_\CC$ with $(\sigma.\sigma)=0$ and $(\sigma.\bar\sigma)>0$.
The two classes ${\rm Re}(\sigma)$ and ${\rm Im}(\sigma)$, which span the positive real plane $$P_T=\RR\cdot{\rm Re}(\sigma)\oplus\RR\cdot{\rm Im}(\sigma),$$
are orthogonal to each other $({\rm Re}(\sigma).{\rm Im}(\sigma))=0$ and of the same norm $({\rm Re}(\sigma))^2=({\rm Im}(\sigma))^2$.
The plane $P_T$ will be considered with its natural orientation given by its base ${\rm Re}(\sigma),{\rm Im}(\sigma)$.

We recall that the Hodge structure $T$ is called \emph{irreducible} if there is no proper subvector space $T'\subset T$ with $T^{2,0}\subset T'_\CC$. This implies that the intersection $T\cap T^{1,1}$ in $T_\CC$ is trivial.
If the signature of $T$ is $(2,r-2)$, the converse holds as well: If $T\cap T^{1,1}=0$, then $T$ is irreducible.


\begin{remark}
Note that any $\QQ$-linear subspace $T'\subset T$ with $T^{2,0}\subset T'_\CC$ is a sub-Hodge structure, i.e.\ the inclusion
$(T'_\CC\cap T^{2,0})\oplus (T'_\CC\cap T^{0,2})\oplus (T'_\CC\cap T^{1,1})\subset T'_\CC$ is an equality.
Indeed, in this case $T'_\CC\cap T^{2,0}=T^{2,0}$ and, applying complex conjugation,  one then also has $T'_\CC\cap T^{0,2}=T^{0,2}$.
This shows that for any $\gamma\in T'_\CC\subset T_\CC$ the classes $\gamma^{2,0},\gamma^{0,2}$, and $\gamma^{1,1}=\gamma-\gamma^{2,0}-\gamma^{0,2}$ are contained in $T'_\CC$. Also observe that for a $\QQ$-linear subspace $T'\subset T$ the condition
$T^{2,0}\subset T'_\CC$ is equivalent to $T^{0,2}\subset T'_\CC$.
\end{remark}

\begin{remark}
For the very general Hodge structure the positive plane $P_T\subset T_\RR$ does not contain non-trivial rational classes, i.e.\ $P_T\cap T=0$. More precisely, the Hodge structures (\ref{eqn:HS}) are parametrized by a period domain of complex dimension $r-2$ and the set of those for which  $P_T\cap T=0$ is the complement of a countable union of proper
closed subsets. 
In the non generic situation, one distinguishes two cases corresponding to the $\QQ$-vector space
$P_T\cap T$ being of dimension one or two:
\begin{enumerate}
\item[(i)] The real plane $P_T$ is defined over $\QQ$, i.e.\ $\dim_\QQ (P_T\cap T)=2$ or, equivalently, $( P_T\cap T)\otimes_\QQ\RR\cong P_T$. For irreducible $T$, the condition is equivalent to $\dim_\QQ T=2$, i.e.\ $P_T =T_\RR$.
\item[(ii)] The real plane $P_T$ contains exactly one rational line, i.e.\ $\dim_\QQ (P_T\cap T)=1$.
\end{enumerate}
Geometrically, (i) corresponds to K3 surfaces (or hyperk\"ahler manifolds) of maximal Picard number. The second
case (ii) does occur, but it is not expected to admit an algebro-geometric interpretation as Hodge structures of this
type are parametrized by a countable union of real Lagrangians \cite{HuyAnn}.
\end{remark}

\begin{lem}\label{lem:irredHS1}
For an irreducible Hodge structure $T$ of K3 type, the orthogonal projection defines an injection
$$T\,\hookrightarrow P_T.$$
In other words, if $T$ is irreducible, $0\ne \sigma\in T^{2,0}$, and $0\ne\gamma\in T$, then $(\sigma.\gamma)\ne0$.
\end{lem}

\begin{proof}
For any class $ \gamma\in T$ in the kernel of the orthogonal projection
$T\to P_T$,  one has $P_T\subset\gamma^\perp_\RR$. Hence,
$0\ne T'\coloneqq\gamma^\perp\subset T$ is a sub-Hodge structure. Irreducibility of $T$ then implies $T'=T$ and, therefore,
$\gamma=0$.
\end{proof}
\subsection{The period field $k_T$}\label{sec:ks}
The \emph{period} of the Hodge structure $T$ is the point $$x_T\coloneqq [T^{2,0}]\in\PP(T_\CC)=\PP(T)\times_\QQ\CC.$$
We will denote
its residue field by $$k_T\coloneqq k(x_T)$$ and call it the \emph{period field} of the Hodge structure $T$. 

To make this more concrete, fix a basis
$\gamma_1,\ldots,\gamma_r$ of the $\QQ$-vector space $T$ and consider the induced isomorphism
$T\congpf\QQ^r$, $\gamma\mapsto ((\gamma.\gamma_i))$. This is the composition of $T\congpf T^*$, $\gamma\mapsto (\gamma.~)$,
or, equivalently, $\gamma\mapsto \sum (\gamma.\gamma_i)\,\gamma_i^*$, and the
natural isomorphism $T^*\congpf\QQ^r$ given by the dual basis $(\gamma_i^*)$.
The $\CC$-linear extension $T_\CC\congpf \CC^r$ maps
a generator $\sigma$ of $T^{2,0}$ to $(x_1\coloneqq(\sigma.\gamma_1),\ldots,x_r\coloneqq(\sigma.\gamma_r))$ and if $x_1\ne0$  then
$$k_T=\QQ(x_2/x_1,\ldots, x_r/x_1)\subset\CC.$$

Lemma \ref{lem:irredHS1} then yields the following consequence for irreducible Hodge structures.

\begin{cor}\label{cor:irredHS}
Let $T$ be an irreducible Hodge structure of K3 type with a fixed basis $(\gamma_i)$ and $0\ne\sigma\in T^{2,0}$.
\begin{itemize}
\item[{\rm (i)}]The coordinates $x_i\coloneqq (\sigma.\gamma_i)\in \CC$ are linearly independent
over $\QQ$.  In particular, $x_i\ne0$ for all $i$.
\item[{\rm (ii)}] The affine coordinates $x_i/x_1$ are linearly independent over $\QQ$, i.e.\ $\bigoplus_{i=1}^r\QQ\cdot(x_i/x_1)\,\hookrightarrow k_T$.
\end{itemize}
In particular, $[k_T:\QQ]\geq \dim_\QQ T$.\qed
\end{cor}

\subsection{The endomorphism field $K_T$}\label{sec:EndoT} From now on until the end of Section \ref{sec:PeriodsHodge} and whenever $K_T$ is used, the bilinear form $(~.~)$ on $T$ is assumed to have
signature $(2,r-2)$. Moreover, we shall assume that the Hodge structure $T$ is irreducible.

Consider endomorphisms $\varphi\colon T\to T$ of the Hodge structure $T$, i.e.\  $\QQ$-linear maps $\varphi$ whose $\CC$-linear extensions
$\varphi_\CC$ (or, simply, $\varphi$) satisfy $\varphi_\CC(T^{p,q})\subset T^{p,q}$. They form an algebra
$$K_T\coloneqq {\rm End}_{\rm Hdg}(T).$$
As $T$ is irreducible, any non-zero $\varphi$ is in fact an isomorphism, for either its kernel or its image defines
a subspace $T'\subset T$ with $T^{2,0}\subset T'_\CC$. Hence, $K_T$ is a division algebra. 
For the same reason, $\varphi={\rm id}$ if and only if
$\varphi_\CC|_{T^{2,0}}={\rm id}$. Moreover, the image $\varphi_\CC(\sigma)$ of a generator $\sigma\in T^{2,0}$
is always a scalar multiple of $\sigma$, and mapping $\varphi$ to the scalar factor therefore defines an injection
$$K_T\,\hookrightarrow \CC.$$ We will denote the image of $\varphi\in K_T$ under this morphism again by $\varphi$, i.e.\
$\varphi_\CC(\sigma)=\varphi\cdot \sigma$.  This immediately shows that for an irreducible Hodge structure
$T$ of K3 type its endomorphism algebra $K_T$ is a field, an observation originally due to Zarhin \cite{Zarhin}.

\begin{remark}
As subfields of $\CC$, the endomorphism field $K_T$ and the period field $k_T$  of an irreducible Hodge structure $T$
are contained in each other:
\begin{equation}\label{eqn:Ksubsetk}
K_T\subset k_T.
\end{equation}
Indeed, for $\varphi\in K_T$ with $\varphi(\sigma)=\varphi\cdot\sigma$, one has $\varphi\cdot(\sigma.\gamma_1)=(\varphi(\sigma).\gamma_1)=(\sigma.\varphi'(\gamma_1))$, where the adjoint $\varphi'$ maps $\gamma_1$ to a $\QQ$-linear combination of 
the $\gamma_i$.
Dividing by $(\sigma.\gamma_1)$ yields (\ref{eqn:Ksubsetk}).

Note that $K_T$ is always a number field (of degree at most $\dim_\QQ T$), whereas  for the
very general Hodge structure on $T$ the period field $k_T$ is of transcendence degree $r-2$. Also note that frequently $K_T\subset\RR$, whereas this is never the case for $k_T$, because
it would contradict the two conditions $(\sigma.\sigma)=0$ and $(\sigma.\bar\sigma)>0$.
\end{remark}

As shown by Zarhin \cite{Zarhin}, $K_T$ is either a totally real field 
or a CM field. Recall that a number field $K$ is totally real if the image of any embedding $K\,\hookrightarrow\CC$ is contained in
$\RR\subset \CC$. It is a CM field if it is a quadratic extension $K=K_0(\sqrt \lambda)$ of a totally real
field $K_0$ such that
$\lambda\in \RR_{<0}$ for any embedding $K_0\,\hookrightarrow \CC$.

\begin{remark}\label{rem:Hodgeadjoint}
(i) Complex conjugation in $\CC$ corresponds
to taking adjoints with respect to $(~.~)$, i.e.\ $(\psi(\gamma_1).\gamma_2)=(\gamma_1.\varphi(\gamma_2))$ for all $\gamma_1,\gamma_2\in T$ if and only if
$\bar\varphi= \psi$ for the images of $\varphi,\psi$ under one, or equivalently any, embedding $K_T\,\hookrightarrow \CC$, cf.\ \cite{Zarhin} or \cite[Ch.\ 3]{HuyK3}. Note that with $\varphi$ also its adjoint
$\psi$ is a Hodge endomorphism and that
$\varphi$ is an isometry if and only if its image in $\CC$  has norm one. In other words, $(\varphi(\gamma_1).\varphi(\gamma_2))=(\gamma_1.\gamma_2)$
for all $\gamma_1,\gamma_2\in T$ if and only if
 $\bar\varphi\cdot\varphi=1$.

(ii) We will need a slightly stronger variant of the last fact. Assume $\varphi,\psi\colon T\to T$ are just $\QQ$-linear maps such that
$\varphi(\sigma)=\lambda\cdot\sigma$ and $\psi(\sigma)=\bar\lambda\cdot\sigma$ for some $\lambda\in\CC$. Then $\psi$
is the adjoint $\varphi'$ of $\varphi$, i.e.\
$(\psi(\gamma_1).\gamma_2)=(\gamma_1.\varphi(\gamma_2))$ for all $\gamma_1,\gamma_2\in T$. To prove this, let $T'\subset T$
be the kernel of $\psi-\varphi'$. Then $T^{2,0}\subset T'_\CC$ if $((\psi-\varphi')(\sigma).\gamma)=0$ for all $\gamma\in T$ or, equivalently,
$\bar\lambda\cdot (\sigma.\gamma)=(\psi(\sigma).\gamma)=(\sigma.\varphi(\gamma))$ for all $\gamma\in T$. However,
as $\varphi(\bar\sigma)=\bar\lambda\cdot\bar\sigma$, the subspace $T''\subset T$ of all such $\gamma$ 
satisfies $T^{0,2}\subset T''_\CC$. By the irreducibility of $T$, this yields $T''=T$ and, therefore, $T'=T$, i.e.\ $\psi=\varphi'$.

(iii) It may be interesting to study the possibly larger number field $L_T$ of all $\QQ$-linear endomorphisms  $\varphi$ of $T$ with just $\varphi(T^{2,0})\subset T^{2,0}$.
It is indeed a field, but it  may not be closed under complex conjugation as those morphisms do not necessarily preserve
$T^{1,1}$.
\end{remark}

\begin{definition}
An irreducible Hodge structure $T$ of K3 type has \emph{complex multiplication (CM)} if its endomorphism ring $K_T={\rm End}_{\rm Hdg}(T)$ is a CM field
and $\dim_{K_T}T=1$.
\end{definition}

Sometimes the less restrictive notion is used, where $K_T$ is required to be a CM field but one allows
$\dim_{K_T}T>1$.
Note that the condition $\dim_{K_T}T=1$ can equivalently be phrased as $[K_T:\QQ]=\dim_\QQ T$.

\begin{remark}
Any CM field $K$ admits a primitive element of norm one, i.e.\ $K=\QQ(\alpha)$ with $\bar\alpha\cdot\alpha=1$.
For $K=K_T$ 
this can be rephrased by saying that whenever $T$ has CM then there exists a Hodge isometry $\alpha\colon T\congpf T$, i.e.\
one in addition has $(\alpha(\gamma_1).\alpha(\gamma_2))=(\gamma_1.\gamma_2)$ for all $\gamma_1,\gamma_2\in T$, such that
every Hodge endomorphism $\varphi\colon T\to T$ can be written as $\varphi=\sum_{i=1}^{r}a_i\alpha^{i-1}$, $a_i\in \QQ$, cf.\ \cite[Thm.\ 3.3.7]{HuyK3} for an elementary proof and references.
The primitive element $\alpha$ has degree $r=\dim_\QQ T$ over $\QQ$. 
\end{remark}

\begin{remark}
In the totally real case, the only elements $\varphi\in K_T$  that correspond to Hodge isometries of $T$ are $\varphi=\pm1$.
Note also that in the totally real case, the analogue of the condition $\dim_{K_T}T=1$ is never realized,
see \cite[Lem.\ 3.2]{vGeemen}: If $K_T$ is a totally real field, then $\dim_{K_T}T\geq3$. Roughly,
if $\dim_{K_T}T=1$, then $K_T\otimes_\QQ\RR$ splits as $\bigoplus V_\varepsilon$,
where $\varepsilon$ runs through all embeddings $K_T\,\hookrightarrow \RR$ and hence
$\dim_\RR V_\varepsilon=1$. On the other hand, the real field $K_T$ acts on the plane
$P_T$ by multiplication, which yields the contradiction $\dim_\RR V_{\rm id}=2$. In the case of a totally
real field acting on a two-dimensional $T$, Zarhin's classification of 
the possible Mumford--Tate groups implies that there must exist an action of a further quadratic extension.
For example, any Hodge structure of K3 type $T$ with $\dim_\QQ T=2$ has complex multiplication,
a Hodge isometry can be written down explicitly, cf.\ \cite[Rem.\ 3.3.10]{HuyK3}.
\end{remark}

\subsection{Complex multiplication} It turns out that complex multiplication can be
read off not only from the endomorphism field
$K_T$ but also from its period field.

\begin{lem}\label{lem:FirstCMkK}
Assume $T$ is an irreducible Hodge structure of K3 type with complex multiplication.
\begin{itemize}
\item[{\rm (i)}] Then the endomorphism field $K_T$ and the period field $k_T$ coincide  (as subfields of $\CC$): $$K_T= k_T.$$
\item[{\rm (ii)}] For any basis $(\gamma_i)$ of $T$ and any  $0\ne\sigma\in T^{2,0}$ the coordinates $x_i\coloneqq(\sigma.\gamma_i)$ satisfy
$$K_T=k_T=\bigoplus_{i=1}^r\QQ\cdot (x_i/x_1).$$
\end{itemize}
\end{lem}

\begin{proof}
Pick a primitive element of norm one, i.e.\ we write $K_T=\QQ(\alpha)$ with a
Hodge isometry $\alpha$. Applying
powers of $\alpha$ to a fixed $0\ne\gamma_1\in T$ yields a basis of $T$. More precisely, as $\deg(\alpha)=\dim_\QQ T$,
the elements $\gamma_i\coloneqq \alpha^{1-i}(\gamma_1)$, $i=1,\ldots ,r$,
form a basis of $T$. As $\alpha$ is an isometry, one has $(\sigma.\gamma_i)=(\sigma. \alpha^{1-i}(\gamma_1))=
(\alpha^{i-1}(\sigma).\gamma_1)=\alpha^{i-1}\cdot(\sigma.\gamma_1)$. Hence, $(\sigma.\gamma_i)/(\sigma.\gamma_1)=\alpha^{i-1}$
and, therefore, $k_T= K_T$. 

To see (ii), we use Corollary \ref{cor:irredHS}. We know that  $x_1\ne0$, hence the $x_i/x_1$ are well defined, and that the
$x_i/x_1$, $i=1,\ldots,r$ are linearly independent over $\QQ$. As $r=[K_T:\QQ]=[k_T:\QQ]$, this proves the second assertion.
\end{proof}

\begin{remark}\label{rem:refereeremark}
Following a suggestion of a referee, we recast the last lemma as follows:
Pick $\sigma\in T^{2,0}$ such that $(\sigma.x_1)=1$. Then
\begin{equation}\label{eqn:refeqn}
T\congpf \bigoplus \QQ\cdot x_i=\bigoplus \QQ\cdot(x_i/x_1),~\gamma\mapsto (\sigma.\gamma)
\end{equation}
is an isomorphism of $\QQ$-vector spaces, because for $\gamma=\sum a_i\gamma_i$, $a_i\in\QQ$,
one finds $(\sigma.\gamma)=\sum a_i x_i$. Furthermore, if $\varphi\in K_T=\End_{\rm Hdg}(T)$ corresponds to $\lambda\in\CC$, i.e. $\varphi(\sigma)=\lambda\cdot\sigma$, then multiplication with $\lambda$
preserves $\bigoplus\QQ\cdot x_i$ and under (\ref{eqn:refeqn}) corresponds to the transpose $\varphi'$.
\end{remark}

The following result is a partial converse. 
Using the above notation, we know that by Corollary \ref{cor:irredHS} 
the natural map $\bigoplus \QQ\cdot (x_i/x_1)\to k_T=\QQ(x_2/x_1,\ldots,x_r/x_1)$ is injective.


\begin{lem}\label{lem:converseCMperiod}
Let $T$ be an irreducible Hodge structure of K3 type with period field $k_T$. 
Assume $L\subset k_T\cap\RR$ is a subfield with $[L:\QQ]\geq (1/2)\dim_\QQ T$ and such that multiplication
with elements in $L$ preserves the subspace $\bigoplus\QQ\cdot (x_i/x_1)\subset k_T$. Then $T$ has complex multiplication.
\end{lem}

\begin{proof}  In most of the argument we only use $L=\bar L$ and $L\subset k_T$. Only at the very
end, $L\subset\RR$ becomes important.

We claim that every $\lambda\in k_T\cap\RR$ that preserves $\bigoplus\QQ\cdot(x_i/x_1)$
is contained in $K_T=\End_{\rm Hdg}(T)$.
Indeed, there exists a unique invertible matrix $(b(\lambda)_{ij})$, $b(\lambda)_{ij}\in\QQ$, with 
$\lambda \cdot x_i=\sum b(\lambda)_{ij}\, x_j$. Using the isomorphism $T\congpf \QQ^r$, $\gamma\mapsto ((\gamma.\gamma_i))$,
we interpret $(b(\lambda)_{ij})$ as a linear map $\varphi\colon T\to T$. Its $\CC$-linear extension satisfies
$$\begin{array}{rcl}
(\varphi(\sigma).~)&=&\sum_j x_j\left(\sum_{i\phantom{j}}\!\!b(\lambda)_{ij}\,\gamma_i^*\right)
=\sum_i\left(\sum_j b(\lambda)_{ij}\,x_j\right)\,\gamma_i^*\\
&=&\sum_i(\lambda\cdot x_i)\,\gamma_i^*=\lambda\cdot\sum_ix_i\,\gamma_i^*,
\end{array}$$
 i.e.\ $\varphi(\sigma)=\lambda\cdot \sigma$.  This is enough to conclude that $\varphi$ is an endomorphism of Hodge structures.
Indeed, if $\gamma\in T^{1,1}\cap T_\RR$, then $(\sigma.\varphi(\gamma))=(\varphi'(\sigma).\gamma)=\bar\lambda\cdot(\sigma.\gamma)=0$, as by (ii) in Remark \ref{rem:Hodgeadjoint}  the adjoint
$\varphi'$ is given by $(b(\bar\lambda)_{ij})$, where we use $\bar\lambda\in L$, and hence $\varphi(\gamma)\in T^{1,1}$.
Therefore, $\varphi\in K_T$ and $\varphi$ is mapped  to $\lambda$ under the natural inclusion $K_T\,\hookrightarrow \CC$, cf.\ Remark \ref{rem:refereeremark}.

Hence, $L\subset K_T$ as subfields of $\CC$ and, therefore,  $[K_T:\QQ]\geq (1/2)\dim_\QQ T$ or, equivalently, $\dim_{K_T}T\leq 2$.
 However, according to \cite[Lem.\ 3.2]{vGeemen}, if $K_T$ is a totally real field, then $\dim_{K_T}T\geq 3$. Hence,  $K_T$ is a CM field. 
 
 Eventually, now using $L\subset\RR$, the inclusion $L\subset K_T$ is in fact
 proper, and, therefore, $\dim_{K_T} T=1$.
\end{proof}

We summarize the discussion as follows.

\begin{prop}\label{prop:CMKk}
For an irreducible Hodge structure  $T$ of K3 type the following conditions are equivalent:
\begin{enumerate}
\item[{\rm (i)}] $T$ is of CM type, i.e.\ $K_T$ is a CM field and $[K_T:\QQ]=\dim_\QQ T$.
\item[{\rm (ii)}]  $k_T$ is a CM field and $[k_T:\QQ]=\dim_\QQ T$.
\item[{\rm (iii)}] $K_T=k_T$.
\end{enumerate}
\end{prop}

\begin{proof}
Let us assume (ii). Then, for dimension reasons,  $\bigoplus\QQ\cdot (x_i/x_1)\,\hookrightarrow k_T$ is a bijection.
In particular, the space  is preserved by multiplication with elements in the maximal totally real subfield 
$L\coloneqq k_T\cap \RR$. Now apply Lemma \ref{lem:converseCMperiod}, using $[k_T\cap\RR:\QQ]=(1/2)\dim_\QQ T$ for the CM field $k_T$.
Hence, (ii) implies (i) and  (iii). 
As by virtue of Lemma \ref{lem:FirstCMkK}, (i) implies (ii) and (iii), it remains to prove that (iii) implies (i) or (ii). However, (iii) together with  $[K_T:\QQ]\leq r\leq [k_T:\QQ]$ yields $\dim_{K_T}T=1$ and, as above, this proves that $K_T$ is a CM field.
\end{proof}

\section{Hodge theory of the twistor space}\label{sec:HodgeTwistor}

As before, $T$ denotes an irreducible Hodge structure of K3 type with signature $(2,r-2)$.
Its positive real plane is denoted $P_T=\langle{\rm Re}(\sigma),{\rm Im}(\sigma)\rangle_\RR\subset T_\RR$, where $0\ne \sigma\in T^{2,0}$.
Associated with $T$ and an abstract class $\ell$ of positive square, there exists a sphere of related Hodge structures.

\subsection{The twistor construction}\label{sec:TwistorConstrT}
For any $d\in \ZZ_{>0}$ we extend $T$ to the Hodge structure of K3 type
$T\oplus\QQ\cdot\ell$ by declaring $\ell$ to be of type $(1,1)$, 
orthogonal to $T$, and to satisfy $(\ell.\ell)=d$. Note that then $P_T\oplus\RR\cdot\ell\subset T_\RR\oplus\RR\cdot\ell$ is a positive three-space.

The associated \emph{twistor base} $\PP^1_\ell\subset\PP(T^{2,0}\oplus T^{0,2}\oplus \CC\cdot\ell)=\PP(P_{T\CC}\oplus\CC\cdot\ell)\subset\PP(T_\CC\oplus\CC\cdot\ell)$ is the conic
$$\PP^1_\ell\coloneqq\{z\mid (z.z)=0\}.$$
Any $z\in \PP^1_\ell$ defines a Hodge structure of K3 type on $T\oplus\QQ\cdot\ell$. Its $(2,0)$-part is the line corresponding to $z$,
 the complex conjugate of the line is the $(0,2)$-part, and the $(1,1)$-part is given as the orthogonal complement of the former two.

Mapping $z\in\PP^1_\ell$ to the oriented, positive real plane $P_z\coloneqq\langle{\rm Re}(z),{\rm Im}(z)\rangle_\RR\subset P_T\oplus\RR\cdot\ell$
yields the usual identification $\PP^1_\ell\cong{\rm Gr}^{\rm po}(P_T\oplus\RR\cdot\ell)$ with the Grassmannian of oriented, positive planes. 
The complex conjugate $\bar z$ corresponds to the same plane with the reversed orientation $P_{\bar z}=\bar P_z\coloneqq\langle{\rm Im}(z),{\rm Re}(z)\rangle_\RR$. We will consider the period point of $T$ and its complex conjugate  $x\coloneqq x_T,\bar x=\bar x_T\in\PP(T_\CC)$
as points in $\PP^1_\ell$ via the natural inclusion $\PP(T_\CC)\subset\PP(T_\CC\oplus\CC\cdot\ell)$.

Thinking of $P_z$ with its orientation being given as the orthogonal complement of a generator $\alpha_z$ of the line $P_z^\perp\subset P_T\oplus\RR\cdot\ell$ provides a natural identification
\begin{equation}\label{eqn:PGRS}
\PP^1_\ell\cong {\rm Gr}_2^{\rm po}(P_T\oplus\RR\cdot\ell)\cong S^2_\ell\coloneqq\{ \alpha \mid (\alpha.\alpha)=1\}\subset P_T\oplus\RR\cdot\ell.
\end{equation}
With this identification, $x$ and $\bar x$ correspond to (the normalization of) $\ell$ and $-\ell$. We think of them 
 as the north and south pole of $S_\ell^2$. 
The \emph{equator} is the circle
$$S^1_\ell\coloneqq\{z\in \PP^1_\ell\mid\ell\in P_z\}\cong\{\alpha\mid(\alpha.\ell)=0\}\subset S_\ell^2.$$

 If for $z\in\PP(P_{T\CC}\oplus\CC\cdot\ell)$ we write $z=[a\sigma+b\bar\sigma+c\ell]$,  $a,b,c\in\CC$,
then $z\in \PP^1_\ell$ if and only if
\begin{equation}\label{eqn:quad}
2ab(\sigma.\bar\sigma)+c^2 d=0.
\end{equation} The only points with $c=0$ are the north and the south pole. For all other points,  $c\ne0$ and,
after scaling, we may assume $c=1$.

 \subsection{Picard number jumping}\label{sec:Picardjumping}
The \emph{Picard number} $\rho_z$ of the Hodge structure on $T\oplus\QQ\cdot\ell$ corresponding to $z\in\PP^1_\ell$ is defined to be
 the dimension of the $\QQ$-vector space of classes of type $(1,1)$, i.e.\ $$\rho_z\coloneqq\dim_\QQ(P_z^\perp\cap(T\oplus\QQ\cdot\ell)).$$

If the original Hodge structure $T$ was irreducible, then $P_x^\perp\cap(T\oplus\QQ\cdot\ell)=\QQ\cdot\ell$ and, therefore,
$\rho_x=1$. For very general $z\in\PP^1_\ell$,
the corresponding Hodge structure on $T\oplus\QQ\cdot\ell$ satisfies $\rho_z=0$. Indeed,
otherwise there would be a rational class in $T\oplus\QQ\cdot\ell$ orthogonal to $z$, but there
are only countable many such classes.
The Hodge structure on 
$T\oplus\QQ\cdot\ell$ corresponding to $z\in\PP^1_\ell$ satisfies $\rho_z\geq1$ (and, therefore,
is not irreducible) if and only if there exists a non-zero
class $\ell'\in T\oplus\QQ\cdot\ell$  orthogonal to $z$. If $z$ is different from $x$ and $\bar x$ and is written in the form
$z=[a\sigma+b\bar\sigma+\ell]$, then orthogonality of $z$ and $\ell'$ 
is expressed by the equation 
\begin{equation}\label{eqn:linear}
a(\sigma.\ell')+b(\bar\sigma.\ell')+(\ell.\ell')=0.
\end{equation}

 The following observation will be made more more precise in the CM case later.
 
 \begin{lem}\label{lem:QbarQbar}
 Let $k_z$ be the period field of a point $z\in \PP^1_\ell$ with
  $\rho_z>0$. Then  $k_T\subset\bar\QQ$ if and only if $k_z\subset\bar\QQ$.
 \end{lem}
 
 \begin{proof}
 If $k_T$ is algebraic, then after rescaling $\sigma$, we may assume $\sigma\in T_{\bar\QQ}\subset T_{\bar\QQ}\oplus\bar\QQ\cdot\ell$. Then the twistor conic
 is defined by the quadratic equation $2ab(\sigma.\bar\sigma)+c^2d=0$ with
 coefficients $(\sigma.\bar\sigma)$ and $d$ contained in $\bar\QQ$.
 Clearly, the intersection with the rational hyperplane $\ell'^\perp$ consists of two points
 with affine coordinates in $\bar\QQ$, one of which corresponds to $z$.
 \end{proof}

\begin{prop}\label{prop:Picardjump}
Assume $T$ is an irreducible Hodge structure of K3 type and signature $(2,r-2)$. Then for the twistor base $\PP^1_\ell\cong S^2_\ell$ one has
\begin{enumerate}
\item[{\rm (i)}]  The set $\{z\mid\rho_z\geq1\}\subset\PP^1_\ell\cong S^2_\ell$ is countable and dense (in the classical topology).
\item[{\rm (ii)}] The set $\{z\mid\rho_z>1\}$ is at most countable and contained in the equator $S^1_\ell\subset S^2_\ell$. 
\end{enumerate}
\end{prop}

\begin{picture}(100,140)
\put(145,19){\begin{tikzpicture}  
\filldraw[draw=black,fill=white,opacity=0.9] (2,2) circle (1.9cm);\end{tikzpicture}}
 
\put(145,58){\begin{tikzpicture}
 \draw[black, line width=0.2mm]  (0,0) .. controls (1.2,-0.5) and (2.6,-0.5) .. (3.8,0);
\end{tikzpicture}}

\put(145,72.5){\begin{tikzpicture}
 \draw[black, line width=0.1mm,opacity=0.4]  (0,0) .. controls (1.2,0.5) and (2.6,0.5) .. (3.8,0);
\end{tikzpicture}}
\put(122,115){$S^2_\ell$}
\put(122,71){$S^1_\ell$}
\put(152,91){\begin{tikzpicture}\draw (-2,1,5) node[anchor=south] {.};\end{tikzpicture}}
\put(172,111){\begin{tikzpicture}\draw (-2,1,5) node[anchor=south] {.};\end{tikzpicture}}
\put(192,119){\begin{tikzpicture}\draw (-2,1,5) node[anchor=south] {.};\end{tikzpicture}}
\put(210,85){\begin{tikzpicture}\draw (-2,1,5) node[anchor=south] {.};\end{tikzpicture}}
\put(192,81){\begin{tikzpicture}\draw (-2,1,5) node[anchor=south] {.};\end{tikzpicture}}
\put(182,51){\begin{tikzpicture}\draw (-2,1,5) node[anchor=south] {.};\end{tikzpicture}}
\put(172,81){\begin{tikzpicture}\draw (-2,1,5) node[anchor=south] {.};\end{tikzpicture}}
\put(182,51){\begin{tikzpicture}\draw (-2,1,5) node[anchor=south] {.};\end{tikzpicture}}
\put(192,81){\begin{tikzpicture}\draw (-2,1,5) node[anchor=south] {.};\end{tikzpicture}}
\put(222,51){\begin{tikzpicture}\draw (-2,1,5) node[anchor=south] {.};\end{tikzpicture}}
\put(232,81){\begin{tikzpicture}\draw (-2,1,5) node[anchor=south] {.};\end{tikzpicture}}
\put(232,91){\begin{tikzpicture}\draw (-2,1,5) node[anchor=south] {.};\end{tikzpicture}}
\put(222,101){\begin{tikzpicture}\draw (-2,1,5) node[anchor=south] {.};\end{tikzpicture}}
\put(212,111){\begin{tikzpicture}\draw (-2,1,5) node[anchor=south] {.};\end{tikzpicture}}
\put(182,51){\begin{tikzpicture}\draw (-2,1,5) node[anchor=south] {.};\end{tikzpicture}}
\put(163,45){\begin{tikzpicture}\draw (-2,1,5) node[anchor=south] {.};\end{tikzpicture}}
\put(147,55){\begin{tikzpicture}\draw (-2,1,5) node[anchor=south] {.};\end{tikzpicture}}
\put(222,51){\begin{tikzpicture}\draw (-2,1,5) node[anchor=south] {.};\end{tikzpicture}}
\put(183,45){\begin{tikzpicture}\draw (-2,1,5) node[anchor=south] {.};\end{tikzpicture}}
\put(167,55){\begin{tikzpicture}\draw (-2,1,5) node[anchor=south] {.};\end{tikzpicture}}
\put(182,51){\begin{tikzpicture}\draw (-2,1,5) node[anchor=south] {.};\end{tikzpicture}}
\put(163,45){\begin{tikzpicture}\draw (-2,1,5) node[anchor=south] {.};\end{tikzpicture}}
\put(147,55){\begin{tikzpicture}\draw (-2,1,5) node[anchor=south] {.};\end{tikzpicture}}
\put(242,51){\begin{tikzpicture}\draw (-2,1,5) node[anchor=south] {.};\end{tikzpicture}}
\put(203,45){\begin{tikzpicture}\draw (-2,1,5) node[anchor=south] {.};\end{tikzpicture}}
\put(187,55){\begin{tikzpicture}\draw (-2,1,5) node[anchor=south] {.};\end{tikzpicture}}

\put(182,71){\begin{tikzpicture}\draw (-2,1,5) node[anchor=south] {.};\end{tikzpicture}}
\put(163,68){\begin{tikzpicture}\draw (-2,1,5) node[anchor=south] {.};\end{tikzpicture}}
\put(147,75){\begin{tikzpicture}\draw (-2,1,5) node[anchor=south] {.};\end{tikzpicture}}
\put(222,71){\begin{tikzpicture}\draw (-2,1,5) node[anchor=south] {.};\end{tikzpicture}}
\put(203,68){\begin{tikzpicture}\draw (-2,1,5) node[anchor=south] {.};\end{tikzpicture}}
\put(187,55){\begin{tikzpicture}\draw (-2,1,5) node[anchor=south] {.};\end{tikzpicture}}
\put(202,31){\begin{tikzpicture}\draw (-2,1,5) node[anchor=south] {.};\end{tikzpicture}}
\put(183,28){\begin{tikzpicture}\draw (-2,1,5) node[anchor=south] {.};\end{tikzpicture}}
\put(157,35){\begin{tikzpicture}\draw (-2,1,5) node[anchor=south] {.};\end{tikzpicture}}
\put(231,33){\begin{tikzpicture}\draw (-2,1,5) node[anchor=south] {.};\end{tikzpicture}}
\put(223,28){\begin{tikzpicture}\draw (-2,1,5) node[anchor=south] {.};\end{tikzpicture}}
\put(2017,35){\begin{tikzpicture}\draw (-2,1,5) node[anchor=south] {.};\end{tikzpicture}}

\put(144,66.5){\begin{tikzpicture}\draw (-2,1,5) node[anchor=south] {.};\end{tikzpicture}}
\put(160,61.5){\begin{tikzpicture}\draw (-2,1,5) node[anchor=south] {.};\end{tikzpicture}}
\put(176,58.5){\begin{tikzpicture}\draw (-2,1,5) node[anchor=south] {.};\end{tikzpicture}}
\put(192,57.4){\begin{tikzpicture}\draw (-2,1,5) node[anchor=south] {.};\end{tikzpicture}}
\put(208,58.){\begin{tikzpicture}\draw (-2,1,5) node[anchor=south] {.};\end{tikzpicture}}
\put(224,60.5){\begin{tikzpicture}\draw (-2,1,5) node[anchor=south] {.};\end{tikzpicture}}
\put(240,64.9){\begin{tikzpicture}\draw (-2,1,5) node[anchor=south] {.};\end{tikzpicture}}

\put(232,81){\begin{tikzpicture}\draw (-2,1,5)[opacity=0.4] node[anchor=south] {.};\end{tikzpicture}}
\put(218,101){\begin{tikzpicture}\draw (-2,1,5)[opacity=0.4] node[anchor=south] {.};\end{tikzpicture}}
\put(202,111){\begin{tikzpicture}\draw (-2,1,5)[opacity=0.4] node[anchor=south] {.};\end{tikzpicture}}
\put(172,101){\begin{tikzpicture}\draw (-2,1,5)[opacity=0.4] node[anchor=south] {.};\end{tikzpicture}}
\put(162,91){\begin{tikzpicture}\draw (-2,1,5)[opacity=0.4] node[anchor=south] {.};\end{tikzpicture}}
\put(182,101){\begin{tikzpicture}\draw (-2,1,5)[opacity=0.4] node[anchor=south] {.};\end{tikzpicture}}
\put(192,91){\begin{tikzpicture}\draw (-2,1,5)[opacity=0.4] node[anchor=south] {.};\end{tikzpicture}}
\put(197,125){${\footnotesize\bullet}$}

\put(202,135){ $x$}
\put(280,70){$\rho>\rho_x$}
\put(257,70){$\leftarrow$}
 \end{picture}
 
 \begin{proof}
The first assertion is well known, cf.\ \cite[Prop.\ 6.2.9]{HuyK3}, and in our situation follows
from the observation that  for any non-zero $\ell'\in T\oplus\QQ\cdot\ell$ the two points $z$ and $\bar z$
in $\ell'^\perp\cap\PP^1_\ell$ satisfy $\rho_z\geq1$. We concentrate on the second assertion. Again, the assertion that the set is at most countable is a standard Hodge theoretic fact.
We have to show it is contained in the equator $S^1_\ell$ of all $z\in \PP^1_\ell$ with $\ell\in P_z$. Assume $\ell_i=\ell_i'+\mu_i\cdot\ell\in T\oplus \QQ\cdot\ell$, $i=1,2$, are two classes  orthogonal to $z$, i.e.\ their orthogonal projections $\bar\ell_i=\bar\ell'_i\oplus \mu_i\cdot\ell$ in $P_T\oplus\RR\cdot\ell$ both span 
$P_z^\perp$, i.e.\ $\langle\bar\ell_1\rangle_\RR=P_z^\perp=\langle\bar\ell_2\rangle_\RR$. If $\mu_1\ne0\ne\mu_2$, 
then $\langle\bar\ell_1\rangle_\RR=\langle\bar\ell_2\rangle_\RR$
implies $\mu_2\cdot\bar\ell_1'=\mu_1\cdot\bar\ell_2'$ in $P_T$. As the orthogonal projection $T\,\hookrightarrow P_T$ is injective by Lemma \ref{lem:irredHS1},
we find $\mu_2\cdot\ell_1'=\mu_1\cdot\ell_2'$, i.e.\ $\ell_1,\ell_2$ are linearly dependent and thus $\rho_z=1$.
If only $\mu_1\ne0$ but $\mu_2=0$, then $\bar\ell_1,\bar\ell_2=\bar\ell_2'\in P\oplus\RR\cdot\ell$ cannot be linearly dependent over $\RR$, which contradicts  $\langle\bar\ell_1\rangle_\RR=\langle\bar\ell_2\rangle_\RR$
and shows that the case  $\mu_1\ne0=\mu_2$ does not occur. In the remaining
case $\mu_1=0=\mu_2$, i.e.\ $\ell_1,\ell_2\in T$, clearly $\bar\ell$ is orthogonal to $\bar\ell_i$, i.e.\ $\ell\in P_z=\bar\ell_1^\perp=\bar\ell_2^\perp$ and, therefore, $z\in S^1_\ell$.
\end{proof}


\begin{remark}
Clearly, the inclusion in (ii) above is strict, i.e.\ the very general $z\in S^1_\ell$ will
still satisfy $\rho_z=0$. Also, there may be points $z\in S^1_\ell$ where
$\rho_z$ does jump, but not excessively, i.e.\ $\rho_z=1$. In fact,
there may be no $z\in\PP^1_\ell$ at all with $\rho_z>1$. For example, when $\dim_\QQ T=2$,
then $P_z^\perp\subset T_\RR\oplus\RR\cdot\ell$ can contain at most one rational class up to scaling.
From Remark \ref{rem:irredm2=0} below, one can also deduce that for $\dim_\QQ T\geq5$, there always exists a $z\in\PP^1_\ell$ with $\rho_z>1$.

Note however, that in the main result of this paper we have to exclude all points in the equator $S^1_\ell$ and not only those with an excessive Picard number, cf.\ Lemma \ref{lem:equatorbad} and Corollary \ref{cor:notCM}.
\end{remark}

Let us fix a class $\ell'\in T\oplus\QQ\cdot \ell$  with $(\ell'.\ell')>0$ and consider
 $x'\in\PP^1_\ell$ orthogonal to $\ell'$. In fact, due to (\ref{eqn:quad}) and (\ref{eqn:linear}),
$x'$ is uniquely determined by $\ell'$  up to conjugation. We will use the notation $T'\coloneqq \ell'^\perp$.
Note that the restriction of $(~.~)$ to $T'$ has signature $(2,r-2)$.
  
 Assume now that $x'$ is not contained in  the equator $S^1_\ell$, i.e.\ $\ell'\not\in T$, and that $x'\ne x,\bar x$, i.e.\ $\ell'\not\in\QQ\cdot\ell$.
To simplify notations we introduce $0\ne m\coloneqq(\ell.\ell')\in\ZZ$ and pick $\sigma\in T^{2,0}$ such that $(\sigma.\ell')=1$.
Writing  $x'=[\sigma'\coloneqq a\sigma+b\bar\sigma+\ell]\in\PP^1_\ell$,  then (\ref{eqn:linear}) becomes
\begin{eqnarray}\label{eqn:abm}
a+b+m=0.
\end{eqnarray}

\subsection{The twistor construction in the CM case}\label{sec:TwCM}
With the same assumptions as before, we now additionally assume that $T$ itself
has CM with CM field $K_T=k_T$. Pick a primitive element $\alpha$ of norm one, i.e.\ $K_T=\QQ(\alpha)$ with $\alpha\cdot\bar\alpha=1$.
Furthermore, we consider a basis $(\gamma_i)$ of $T$ that is given by letting $\gamma_1$ be the $T$-component of $\ell'$ and setting
$\gamma_i\coloneqq\alpha^{-1}(\gamma_{i-1})$ for $i>1$. As we have seen earlier (cf.\ proof of Lemma
\ref{lem:FirstCMkK}), then $\sigma$ corresponds to the vector
$(1,\alpha,\ldots,\alpha^{r-1})$ under $T\congpf \QQ^{r}$, $\gamma\mapsto((\gamma.\gamma_i))$.

A basis of $T'= \ell'^\perp$ is then given by $\gamma_i'\coloneqq \gamma_i-m^{-1}(\ell'.\gamma_i)\cdot\ell$.
With respect to this basis, the period $x'$ is represented by $((\sigma'.\gamma_i'))$. The coefficients
can be computed as follows:
\begin{eqnarray*}
(\sigma'.\gamma_i')&=&\left(a\sigma+b\bar\sigma+\ell.\gamma_i-m^{-1}(\ell'.\gamma_i)\cdot\ell\right)\\
&=&a(\sigma.\gamma_i)+b(\bar\sigma.\gamma_i)-dm^{-1}(\ell'.\gamma_i)\\
&=&a\alpha^{i-1}+b\alpha^{1-i}+m_i\\
&=&a(\alpha^{i-1}-\alpha^{1-i})-m\alpha^{1-i}+m_i
\end{eqnarray*}
with $m_i\coloneqq -dm^{-1}(\ell'.\gamma_i)\in\QQ$. Use (\ref{eqn:abm}) for the last equality.
Note that $(\sigma'.\gamma_1)\inÊ\QQ$.

\begin{remark}\label{rem:irredm2=0}
As we assume that $x'\not\in S_\ell^1$, the Hodge structure $T'$ does not contain any non-trivial
rational $(1,1)$-classes and, as the signature of $T'$ is $(2,r-2)$, is in fact irreducible.
\end{remark}

\begin{lem}\label{lem:periodsofTprime}
 The period field of the irreducible Hodge structure $T'$  is described by $$k_{T'}=\QQ(x'_i).$$
Here, $x'_i\coloneqq a(\alpha^{i-1}-\alpha^{1-i})-m\alpha^{1-i}$, $i=1,\ldots,r$, which are all non-zero.
 Moreover, the natural map yields an injection
$\bigoplus\QQ\cdot x_i'\,\hookrightarrow k_{T'}$.
\end{lem}

\begin{proof}
This is a direct consequence of Corollary \ref{cor:irredHS} and the fact that $x'_1\in\QQ$.
\end{proof}

The next technical result together with Lemma \ref{lem:converseCMperiod} will lead to an action of the 
maximal totally real subfield $K_T^0\coloneqq K_T\cap \RR$ of the CM field $K_T$ 
on the Hodge structures $T'=\ell'^\perp$ away from the equator. As $K_T^0$
is rather big, recall it satisfies $[K_T^0:\QQ]=r/2$, this will suffice to conclude.

\begin{lem}\label{lem:KTOPLUS}
The real field $K_T^0\coloneqq K_T\cap \RR$ has the following properties.
\begin{enumerate}
\item[{\rm (i)}] $K_T^0\subset \bigoplus\QQ\cdot x_i'\subset k_{T'}$.
\item[{\rm (ii)}] Multiplication in $k_{T'}$ with elements of $K_T^0$ preserves the subspace
$\bigoplus\QQ\cdot x_i'$.
\end{enumerate}
\end{lem}

\begin{proof} As $0\ne x'_1\in\QQ$, the second assertion implies the first, but for greater clarity we state and prove them separately.

Any element $f\in K_T$ can be uniquely written as $f=\sum_{i=1}^ra_i\alpha^{i-1}$ with $a_i\in\QQ$.
As $|\alpha|=1$, the complex conjugate of $f$ is given by $\bar f=\sum_{i=1}^ra_i\alpha^{1-i}$.
Hence, $f\in K_T^0$ if and only if $\sum_{i=1}^r a_i(\alpha^{i-1}-\alpha^{1-i})=0$. 
In this case, one finds $\sum_{i=1}^ra_ix'_i=a\sum_{i=1}^ra_i(\alpha^{i-1}-\alpha^{1-i})-m\sum_{i=1}^ra_i\alpha^{1-i}=-m\bar f=-mf$, and,
therefore, $f\in \bigoplus\QQ\cdot x'_i$. Hence, $K_T^0\subset \bigoplus\QQ\cdot x'_i$.
\medskip

For any $f=2\sum_{i=1}^ra_i\alpha^{i-1}$, the decomposition in its real and imaginary
part is of the form $f=(1/2)(f+\bar f)+(1/2)(f-\bar f)=\sum_{i=1}^ra_i(\alpha^{i-1}+\alpha^{1-i})+
\sum_{i=1}^ra_i(\alpha^{i-1}-\alpha^{1-i})$. In particular, the elements $\alpha^{i-1}+\alpha^{1-i}$, $i=1,\ldots,r$, generate
$K_T^0$ (but, for dimension reasons, they are not linearly independent). 
Observe that $(\alpha+\alpha^{-1})^{i-1}=\alpha^{i-1}+\alpha^{1-i}+M$, where $M$ is a linear combination of
$\alpha^j+\alpha^{-j}$, $j=0,\ldots, i-2$. In particular, $\alpha+\alpha^{-1}$ is a primitive element of $K_T^0$.
Thus, by induction, in order to show that $K_T^0$ preserves $\bigoplus\QQ\cdot x_i'$,
it suffices to prove that multiplication with $\alpha+\alpha^{-1}$ does.

A computation yields that $(\alpha+\alpha^{-1})x_i'=x'_{i-1}+x'_{i+1}$ for $2\leq i\leq r-1$. 
For $i=1$ observe $(\alpha+\alpha^{-1}) x'_1=-m(\alpha+\alpha^{-1})\in K_T^0\subset \bigoplus\QQ x'_i$ 
using (i).
To deal with the case $i=r$, we write $\alpha^r=\sum_{i=1}^{r}c_i\alpha^{i-1}$ for some $c_i\in\QQ$ and, thus, $\alpha^{-r}=\sum c_i\alpha^{1-i}$.
Hence, $(\alpha+\alpha^{-1})x'_r=x'_{r-1}+a(\alpha^r-\alpha^{-r})-m\alpha^{-r}=x'_{r-1}+\sum_{i=1}^{r}c_ix'_i\in\bigoplus\QQ \cdot x'_i$.
\end{proof}

\begin{remark}
(i) We stress that both assumptions, $m=(\ell.\ell')\ne0$ and $m\in\ZZ$, are used in the proof of the lemma. Indeed,
the equation $a_ix'_i=-mf$ would otherwise not yield the desired inclusion $K_T^0\,\hookrightarrow\bigoplus\QQ\cdot x'_i$.
Similarly, one would not expect $K_T^0$ to be contained in $k_{T'}$ when $T'$ is not irreducible.

(ii) Although the proposition only establishes properties of the totally real field $K_T^0$, its proof uses the action of the CM field $K_T$ on $T$.
It is unclear whether also in the totally real case $K_T=K_T^0$ is a subfield of $k_{T'}$.
\end{remark}

\begin{prop}\label{prop:FinalHodge}
Assume $T$ is an irreducible Hodge structure of K3 type with signature $(2,r-2)$ and
CM by $K_T$. Let $x'\in\PP^1_\ell\setminus S^1_\ell$ be a point in the twistor base
orthogonal to a  class $\ell'\in T\oplus\QQ\cdot\ell$ with $(\ell'.\ell')>0$

Then the corresponding Hodge structure on $T'=\ell'^\perp\subset T\oplus\QQ\cdot\ell$ is
an irreducible Hodge structure of K3 type with signature $(2,r-2)$ and CM.
The maximal totally real subfields of the CM fields $K_T$ and $K_{T'}$ coincide: $K_T^0=K_{T'}^0$.
\end{prop}

\begin{proof}
The irreducibility of $T'$ follows from Proposition \ref{prop:Picardjump}.
 Lemma \ref{lem:KTOPLUS} shows that $K_T^0$ is contained in $k_{T'}\cap \RR$ and that it preserves
the subspace $\bigoplus\QQ\cdot x'_i\subset k_{T'}$. As $2\,[K_T^0:\QQ]=[K_T:\QQ]=\dim_\QQ T=\dim_\QQ T'$, Lemma \ref{lem:converseCMperiod} applies and yields the result.
\end{proof}

The assumption that only fibres away from the equator are allowed is not an artefact of our technique
unless we are in the case $r=\dim_\QQ T=2$.

\begin{lem}\label{lem:equatorbad}
Assume $\dim_\QQ T>2$. Then a Hodge structure $T'=\ell'^\perp$ with $(\ell'.\ell')>0$ for which
the corresponding point $x'$ is contained in the equator $S^1_\ell$ is not of CM type.
In fact, in the case that $T'$ is not irreducible, also the minimal Hodge structure
of K3 type $T''\subset T'$ with $T'^{2,0}\subset T''_\CC$  is not of CM type.
\end{lem}

\begin{proof}
Observe that for such a Hodge structure, $\ell$ is contained in $P_{T'}$.
If $T'$ were of  CM type, then its endomorphism field $K_{T'}$ 
would satisfy $K_{T'}\cdot\ell=T'$. However,
all elements of $K_{T'}$ act as endomorphisms of the Hodge
structure and, in particular, preserve $P_{T'}$. This only leaves the possibility 
that $P_{T'}=T'_\RR$, which yields the contradiction $\dim_\QQ T=\dim_\QQ T'=2$.

If $T'$ is not irreducible, then the same arguments apply to $T''\subset T'$ and prove
$\dim_\QQ T''=2$. As $T'=T''\oplus T''^\perp$ and $T''^\perp$  is pure of type $(1,1)$, the 
residue fields $k_{T'}$ and $k_{T''}$ of the two Hodge structures $T'$ and $T''$ coincide.
To get a contradiction, it is enough to show that $[k_{T''}:\QQ]=[k_{T'}:\QQ]>2$. For this we 
adapt the approach explained before to the case $m=0$. 

The classes
$\ell$ and $\gamma'_i\coloneqq \gamma_i-((\ell'.\gamma_i)/(\ell'.\ell'))\cdot\ell'$, $i=1,\ldots, r$,
generate $T'$. Hence, the period field $k_{T'}$  is generated by $(\sigma'.\ell)=d\in\ZZ$ and $(\sigma'.\gamma_i')=a(\alpha^{i-1}-\alpha^{1-i})$, where we use that
$a+b=0$ and $a=\sqrt{d/(2(\sigma.\bar\sigma))}\in \RR$ under the assumption $m=(\ell.\ell')=0$.
However, the elements $\alpha^j-\alpha^{-j}$ generate the subspace of all purely imaginary elements in $K_T$,
which is of dimension $r/2$. Thus, the elements $a(\alpha^{i-1}-\alpha^{1-i})=(\sigma'.\gamma_i')\in k_{T''}$,
which are all purely imaginary, already span a sub-vector space of dimension $r/2\geq2$. Therefore, one finds
the contradiction $[k_{T'}:\QQ]>2$.
\end{proof}

\subsection{CM fields of twistor fibres}
In the situation of Proposition \ref{prop:FinalHodge}, we have $K_T^0=K_{T'}^0$, but the
totally imaginary quadratic extensions of this field $$K_T^0\subset K_T\text{  and }K_{T'}^0\subset K_{T'}$$
will usually be distinct. This can be made precise as follows. As before, pick a primitive element
of norm one of $K_T$ and write $K_T=\QQ(\alpha)$. Then the maximal totally real
field is $K_T^0=\QQ(\alpha+\alpha^{-1})$, see the arguments in the proof of Lemma \ref{lem:KTOPLUS}.

\begin{cor}\label{cor:CMfieldreconstruction}
Under the assumption of Proposition \ref{prop:FinalHodge} and additionally assuming
$x'\ne x,\bar x$, the endomorphism field $K_{T'}$
of the CM Hodge structure $T'$ is the quadratic extension of $K_T^0=K_{T'}^0$ described by
\begin{equation}\label{eqn:Extn}
X^2+\gamma X+\delta=0
\end{equation}
with $\gamma\coloneqq m(\alpha+\alpha^{-1})\in K_T^0$ and $\delta\coloneqq m^2-d(\sigma.\bar\sigma)^{-1}(\alpha^2+\alpha^{-2}-2)\in K_T^0$.
Here, $d=(\ell.\ell)$, $m=(\ell.\ell')$ and $\sigma\in T^{2,0}$ is chosen such that $(\sigma.\ell')=1$.
\end{cor}

\begin{proof} 
As before, we represent $x'$ by $a\sigma+b\bar\sigma+\ell$. According to
(\ref{eqn:quad}) and (\ref{eqn:abm}), one has $2ab(\sigma.\bar\sigma)+d=0$ and $a+b+m=0$.
From the latter we conclude
 $$K_{T'}=K_{T}^0(a\alpha+b\alpha^{-1}).$$
Indeed, by Lemma  \ref{lem:periodsofTprime} $x_2'= a(\alpha-\alpha^{-1})-m\alpha^{-1}=a\alpha+(b+m)\alpha^{-1}-m\alpha^{-1}=a\alpha+b\alpha^{-1}$ and by the arguments in the proof of Lemma \ref{lem:KTOPLUS}
we know that $\bigoplus \QQ\cdot x_i'=k_{T'}$ as a $K_T^0$-vector space is generated by $x_1' \in \QQ$ and $x_2'$.

Therefore, it suffices to show that $a\alpha+b\alpha^{-1}$ satisfies (\ref{eqn:Extn}). This follows from
a computation using $(a\alpha+b\alpha^{-1})^2=a^2\alpha^2+b^2\alpha^{-2}-d(\sigma.\bar\sigma)^{-1}$
and $m(\alpha+\alpha^{-1})(a\alpha+b\alpha^{-1})=-(a+b)(a\alpha^2+b\alpha^{-2}+(a+b))$.
\end{proof}

It is an exercise to check directly that (\ref{eqn:Extn}) does indeed define a totally imaginary quadratic extension of $K_T^0$, as we know it should. Indeed, $|\alpha|=1$ and $\alpha\ne1$ imply $\alpha^2+\alpha^{-2}<2$
and, therefore, $\gamma^2<4\delta$.

\begin{remark} In (\ref{eqn:Extn}) only $m$ seems to depend on the actual class
$\ell'$ or the point $x'$.
But implicitly $\ell'$ is involved once more via the condition $(\sigma.\ell')=1$.
There are of course many classes $\ell'_1,\ell'_2$ with the same $m$, i.e.\
$(\ell.\ell'_1)=(\ell.\ell'_2)$. However, if $\sigma$ can be chosen such that
for both $(\sigma.\ell'_i)=1$, $i=1,2$, then by virtue of the irreducibility of the Hodge structure
$T$ one has $\ell'_1-\ell'_2\in \QQ\cdot\ell$ and $(\ell.\ell'_1)=(\ell.\ell'_2)$ would in fact show $\ell'_1=\ell'_2$. Hence, (\ref{eqn:Extn}) will be the same only for conjugate pairs
of points $x',\bar x'$. Of course, the CM fields of two even non-conjugate points
can be isomorphic without (\ref{eqn:Extn}) being identical.
\end{remark}

\section{Period values}\label{sec:PV}

So far we have encoded a K3 Hodge structure $T$ by the line $T^{2,0}\subset T_\CC$
or, equivalently, by its period point $x\in \PP(T_\CC)$. In the applications to K3 surfaces with CM, which
are always defined over $\bar\QQ$, a finer structure is present, namely a generator $\sigma\in T^{2,0}$ that
is unique up to scalars in $\bar\QQ^\ast$. In this section we formalize the situation. We introduce the period
value and explain how it behaves in a twistor family.

\subsection{Period value} Let $T$ be an irreducible Hodge structure of K3 type with signature $(2,r-2)$
and let  $\Sigma\subset T^{2,0}$ be a $\bar\QQ$-line, i.e.\ a $\bar\QQ$-linear subspace of dimension one. Then we define the \emph{period value}
$$r_{{\scriptscriptstyle T,\Sigma}}\coloneqq (\sigma.\gamma)\in \CC^\ast/\bar\QQ^\ast.$$
Here, $0\ne\sigma\in\Sigma$ and $0\ne\gamma\in T$ are chosen arbitrarily.

\begin{lem}\label{lem:PVindep}
If $T$ has complex multiplication, then
$r_{\scriptscriptstyle T,\Sigma}\in\CC^\ast/\bar\QQ^\ast$ is well-defined,
i.e.\ it is independent of the choices of $\sigma\in \Sigma$ and $\gamma\in T$.
\end{lem}

\begin{proof}
Indeed, for $\lambda\in\bar\QQ^\ast$ one clearly has $(\lambda\cdot\sigma.\gamma)\equiv(\sigma.\gamma)$
in $\CC^\ast/\bar\QQ^\ast$. To prove independence of $\gamma$, choose a primitive element 
 $\alpha\in K$ of norm one of the CM field $K$ of $T$.
Then any other element $\gamma'\in T$ can be written as $\gamma'=(\sum a_i\alpha^{i})(\gamma)$ with $a_i\in \QQ$. Then $(\sigma.\gamma')=\sum a_i\cdot (\alpha^{-i}(\sigma).\gamma)=(\sum a_i\alpha^{-i})\cdot(\sigma.\gamma)\equiv(\sigma.\gamma)\in\CC^\ast/\bar\QQ^\ast$.
\end{proof}

Clearly, any two $\bar\QQ$-lines $\Sigma_1,\Sigma_2\subset T^{2,0}$ differ by some complex scalar
$\lambda\in\CC^\ast$, i.e.\ $\Sigma_2=\lambda\cdot\Sigma_1$. The effect on the period value is expressed by
$$r_{ \scriptscriptstyle T,\lambda\cdot\Sigma}=\lambda\cdot r_{\scriptscriptstyle T,\Sigma}.$$

In the following we will often write $ r_\sigma\coloneqq r_{\scriptscriptstyle T,\sigma}\coloneqq r_{\scriptscriptstyle T,\bar\QQ\cdot\sigma}$.

\begin{lem}\label{lem:r-1sigma}
 If $T$ has complex multiplication and $0\ne\sigma\in T^{2,0}$, then 
 $$r_\sigma^{-1}\cdot\sigma\in T\otimes\bar\QQ\text{ and } (\sigma.\bar\sigma)\equiv r_\sigma\cdot\bar r_\sigma \text{ in }\CC^\ast/\bar\QQ^\ast.$$
\end{lem}

\begin{proof}
For the first assertion observe that for all $\gamma\in T$ one has $(r_\sigma^{-1}\cdot\sigma.\gamma)\in\bar\QQ$. The second assertion follows from the first.
\end{proof}

\subsection{Period value of twistor fibres}\label{sec:PVTW}
We are using the notation of Section \ref{sec:TwCM}. In particular, $T$ is assumed to have complex multiplication.

 Let $x'\in \PP^1_\ell\setminus S^1_\ell$ be orthogonal to some fixed $\ell'\in T\oplus \QQ\cdot\ell$ with $(\ell'.\ell')>0$. Consider the induced natural irreducible Hodge structure on $T'\coloneqq \ell'^\perp\subset T\oplus\QQ\cdot\ell$ which has again signature $(2,r-2)$.

\begin{prop}
Assume that $x'$ is represented by $\sigma'=a\sigma+b\bar\sigma+c\ell$. Then in $\CC^\ast/\bar\QQ^\ast $
$$b\equiv a\cdot r_\sigma/ \bar r_\sigma \text{ and }
r_{\sigma'}\equiv c\equiv a\cdot r_\sigma.$$
Furthermore,  $r_{\sigma'}=a\cdot r_\sigma=\bar r_\sigma$ if $b=1$ and $r_{\sigma'}\equiv 1$ if $c=1$.
\end{prop}

\begin{proof}
Note that by virtue of Proposition \ref{prop:FinalHodge} the Hodge structure $T'$ has complex multiplication and, therefore, its period values $r_{\sigma'}\coloneqq r_{\scriptscriptstyle T',\bar\QQ\cdot\sigma'}$ are well defined for any $0\ne \sigma'\in  T^{2,0}$.

Let us first look at the case that  $r_\sigma\equiv 1$, i.e.\ $\sigma\in T\otimes\bar\QQ$, and $a=1$.
Then, (\ref{eqn:quad}) and (\ref{eqn:linear}) imply $b,c\in\bar\QQ$, cf.\ Lemma \ref{lem:QbarQbar}. In the general case, rewrite $r_{a\sigma}^{-1}\sigma'$ as
follows $$r_{a\sigma}^{-1}\sigma'=r_{a\sigma}^{-1}(a\sigma)+(b/\bar a)(\bar r_{a\sigma}/r_{a\sigma})\overline{(r_{a\sigma}^{-1} (a\sigma))}+(r_{a\sigma}^{-1}c)\ell.$$
Then by the first step $b\equiv \bar a \cdot r_{a\sigma}/\bar r_{a\sigma}\equiv a\cdot r_\sigma/\bar r_\sigma$  and $c\equiv r_{a\sigma}\equiv a\cdot r_\sigma$ in $\CC^\ast/\bar\QQ^\ast$.

To determine $r_{\sigma'}$ we pick an arbitrary $0\ne\gamma\in T'$ and compute
$(\sigma'.\gamma)=a(\sigma.\gamma)+b\overline{(\sigma.\gamma)}+c(\ell.\gamma)=
a (u \cdot r_\sigma+v\cdot (r_\sigma/\bar r_\sigma)\cdot \bar r_\sigma+ w\cdot r_\sigma)$,
$u,v,w\in\bar\QQ$, which suffices to conclude. 

The assertion for $c=1$ is immediate, and for $b=1$, let
$\gamma\in T$ and compute $r_{\sigma'}\equiv (\sigma'.\gamma)=a(\sigma.\gamma)+\overline{(\sigma.\gamma)}\equiv a\cdot r_\sigma+\bar r_\sigma\equiv a\cdot r_\sigma$.
\end{proof}

In the geometric situation, where all CM fibres $\ks_t$ come equipped with a natural $\bar\QQ$-line
in $H^{2,0}(\ks_t)$, we expect the natural period values to depend on $t$. As the three natural choices
for $\sigma'$ consisting of setting $a=1$, $b=1$, resp.\ $c=1$, all lead to constant period
values $r_{\sigma'}$ no
preferred choice for $\sigma'$ seems to suggest itself from a purely Hodge theoretic point of view.


\section{K3 surfaces with CM}\label{sec:K3CM}
The results of the previous sections can be applied to the transcendental part $$T\coloneqq T(S)\coloneqq{\rm NS}(S)^\perp\otimes\QQ\subset H^2(S,\QQ)$$
of a complex 
projective K3 surface $S$. In fact, everything remains valid if instead of a K3 surface $S$, one considers
a projective hyperk\"ahler manifold, cf.\ Section \ref{sec:HK}. For simplicity we restrict to the case of K3 surfaces and leave the necessary modifications in the hyperk\"ahler case to the reader. 

\subsection{Dictionary}\label{sec:Dict}
The intersection form on $H^2(S,\ZZ)$ provides the bilinear
form $(~.~)$ on $T=T(S)$. It is positive definite on the plane
$$P_T=(H^{2,0}(S)\oplus H^{0,2}(S))\cap H^2(S,\RR)$$
and, by virtue of the projectivity of $S$, negative definite on its orthogonal complement $P_T^\perp\subset T_\RR$.
The \emph{period field} and the \emph{endomorphism field} of $S$ are introduced as the
corresponding fields of the transcendental part $T(S)$:
$$k_S\coloneqq k_{T(S)}\text{ and }K_S\coloneqq K_{T(S)}.$$
So, $k_S$ is the residue field of the period point of $S$ taken either in
$\PP(T(S)_\QQ)\times_\QQ\CC$ or $\PP(H^2(S,\QQ))\times_\QQ\CC$ and for $K_S$ one knows
$\dim_{K_S}T=(22-\rho(S))\cdot[K_S:\QQ]^{-1}$.

Fix an ample class $\ell={\rm c}_1(L)\in H^2(S,\ZZ)$ and consider $T(S)\oplus \QQ\cdot\ell\subset
H^2(S,\QQ)$ with its natural quadratic form. Then $P_T\oplus\RR\cdot\ell\subset H^2(S,\RR)$ is positive
definite of dimension three. As in the abstract setting, one can think of the twistor base
$\PP^1_\ell$ as the set of all $\sigma_t\in H^{2,0}((M,I_t))$ up to scaling or as the set of all oriented positive
planes $P_t=\langle {\rm Re}(\sigma_t),{\rm Im}(\sigma_t)\rangle_\RR\subset P_T\oplus \RR\cdot\ell$
or as the sphere $S_\ell^2\subset P_T\oplus\RR\cdot\ell\subset H^2(S,\RR)$
of K\"ahler classes $\omega_t$ that span the line orthogonal to $P_t$. The Picard number
$\rho_z$ (cf.\ Section \ref{sec:Picardjumping}) of a point $z\in\PP^1_\ell$
 and the Picard number $\rho(\ks_t)$ of the twistor fibre $\ks_t=(M,I_t)$, with $t$ corresponding to  $z$
 under $ S^2_\ell\cong\PP^1_\ell$,
compare as follows $$\rho_z+\rho(S)-1=\rho(\ks_t).$$

The equator $$S^1_\ell\subset S^2_\ell\subset P_T\oplus\RR\cdot\ell\subset H^2(S,\RR)$$ can be thought
of as the set of K\"ahler classes $\omega_t$ orthogonal to $\ell$ or as the set of complex structures $I_t$
such that $\ell$ is contained in the positive plane $P_t$. As an immediate consequence of the abstract
Proposition \ref{prop:Picardjump},  we state the following.
 
 \begin{cor}
 Assume $\ks\to \PP^1$ is the twistor space associated with a polarized K3 surface $(S,L)$.
 If $\rho(\ks_t)>\rho(S)$, then $t$ is contained in the equator $S^1_\ell\subset S^2_\ell$.\qed
 \end{cor}
 
\begin{remark}
Note however that there are cases where $\rho(S)$ is maximal, i.e.\ for all fibres one has
$\rho(\ks_t)\leq\rho(S)$. Clearly, this is the case when $\rho(S)=20$. It would be interesting
to work out geometric conditions on $(S,L)$ such that $\rho(\ks_t)\leq\rho(S)$ holds for all twistor fibres.
\end{remark}

The main Theorem  \ref{thm:main} is then an immediate consequence of Proposition \ref{prop:FinalHodge}.
We rephrase it in the following alternative but equivalent form.

\begin{thm}
Consider the twistor space $\ks\to\PP^1$ associated with a polarized, complex K3 surface
$(S,L)$ with complex multiplication. Then every algebraic fibre $\ks_t$ such that
$\ell={\rm c}_1(L)$ is not contained in $H^{2,0}(\ks_t)\oplus H^{0,2}(\ks_t)$ has complex multiplication
as well. Moreover, the maximal totally real subfields of the two endomorphism fields of $S$ and $\ks_t$
coincide.\qed
\end{thm}

Using Lemma \ref{lem:equatorbad}, one sees that in the case of $\rho(S)<20$ the fibres over the equator have
indeed to be avoided. This is the following corollary.

\begin{cor}\label{cor:notCM}
Assume $(S,L)$ is a polarized K3 surface with complex multiplication and Picard number $\rho(S)<20$.
Then no algebraic fibre $\ks_t$ of the associated twistor space with $t$ contained
in the equator has complex multiplication.
\end{cor}

\begin{proof}
As we assume $\ks_t$ is algebraic, there must exist an $\ell'\in T\oplus\QQ\cdot\ell$ orthogonal to the
period $a\sigma+b\bar\sigma+\ell$ with $(\ell'.\ell')>0$. Hence, Lemma \ref{lem:equatorbad} applies.
\end{proof}

\begin{remark}
The last result in particular shows that no twistor fibre $\ks_t$ of a twistor space associated with a
 polarized K3 surface $(S,L)$ with complex multiplication and Picard number $\rho(S)<20$ will ever have maximal Picard
number $\rho(\ks_t)=20$.
\end{remark}

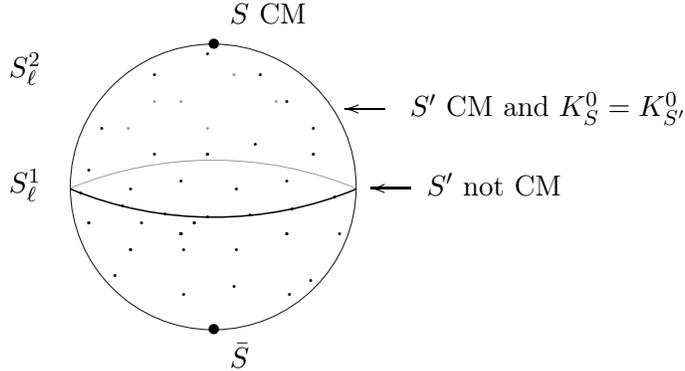
\begin{picture}(100,140)
\put(145,19){\begin{tikzpicture}  
\filldraw[draw=black,fill=white,opacity=0.9] (2,2) circle (1.9cm);\end{tikzpicture}}
 
\put(145,58){\begin{tikzpicture}
 \draw[black, line width=0.2mm]  (0,0) .. controls (1.2,-0.5) and (2.6,-0.5) .. (3.8,0);
\end{tikzpicture}}

\put(145,72.5){\begin{tikzpicture}
 \draw[black, line width=0.1mm,opacity=0.4]  (0,0) .. controls (1.2,0.5) and (2.6,0.5) .. (3.8,0);
\end{tikzpicture}}
\put(122,115){$S^2_\ell$}
\put(122,71){$S^1_\ell$}
\put(152,91){\begin{tikzpicture}\draw (-2,1,5) node[anchor=south] {.};\end{tikzpicture}}
\put(172,111){\begin{tikzpicture}\draw (-2,1,5) node[anchor=south] {.};\end{tikzpicture}}
\put(192,119){\begin{tikzpicture}\draw (-2,1,5) node[anchor=south] {.};\end{tikzpicture}}
\put(210,85){\begin{tikzpicture}\draw (-2,1,5) node[anchor=south] {.};\end{tikzpicture}}
\put(192,81){\begin{tikzpicture}\draw (-2,1,5) node[anchor=south] {.};\end{tikzpicture}}
\put(182,51){\begin{tikzpicture}\draw (-2,1,5) node[anchor=south] {.};\end{tikzpicture}}
\put(172,81){\begin{tikzpicture}\draw (-2,1,5) node[anchor=south] {.};\end{tikzpicture}}
\put(182,51){\begin{tikzpicture}\draw (-2,1,5) node[anchor=south] {.};\end{tikzpicture}}
\put(192,81){\begin{tikzpicture}\draw (-2,1,5) node[anchor=south] {.};\end{tikzpicture}}
\put(222,51){\begin{tikzpicture}\draw (-2,1,5) node[anchor=south] {.};\end{tikzpicture}}
\put(232,81){\begin{tikzpicture}\draw (-2,1,5) node[anchor=south] {.};\end{tikzpicture}}
\put(232,91){\begin{tikzpicture}\draw (-2,1,5) node[anchor=south] {.};\end{tikzpicture}}
\put(222,101){\begin{tikzpicture}\draw (-2,1,5) node[anchor=south] {.};\end{tikzpicture}}
\put(212,111){\begin{tikzpicture}\draw (-2,1,5) node[anchor=south] {.};\end{tikzpicture}}
\put(182,51){\begin{tikzpicture}\draw (-2,1,5) node[anchor=south] {.};\end{tikzpicture}}
\put(163,45){\begin{tikzpicture}\draw (-2,1,5) node[anchor=south] {.};\end{tikzpicture}}
\put(147,55){\begin{tikzpicture}\draw (-2,1,5) node[anchor=south] {.};\end{tikzpicture}}
\put(222,51){\begin{tikzpicture}\draw (-2,1,5) node[anchor=south] {.};\end{tikzpicture}}
\put(183,45){\begin{tikzpicture}\draw (-2,1,5) node[anchor=south] {.};\end{tikzpicture}}
\put(167,55){\begin{tikzpicture}\draw (-2,1,5) node[anchor=south] {.};\end{tikzpicture}}
\put(182,51){\begin{tikzpicture}\draw (-2,1,5) node[anchor=south] {.};\end{tikzpicture}}
\put(163,45){\begin{tikzpicture}\draw (-2,1,5) node[anchor=south] {.};\end{tikzpicture}}
\put(147,55){\begin{tikzpicture}\draw (-2,1,5) node[anchor=south] {.};\end{tikzpicture}}
\put(242,51){\begin{tikzpicture}\draw (-2,1,5) node[anchor=south] {.};\end{tikzpicture}}
\put(203,45){\begin{tikzpicture}\draw (-2,1,5) node[anchor=south] {.};\end{tikzpicture}}
\put(187,55){\begin{tikzpicture}\draw (-2,1,5) node[anchor=south] {.};\end{tikzpicture}}

\put(182,71){\begin{tikzpicture}\draw (-2,1,5) node[anchor=south] {.};\end{tikzpicture}}
\put(163,68){\begin{tikzpicture}\draw (-2,1,5) node[anchor=south] {.};\end{tikzpicture}}
\put(147,75){\begin{tikzpicture}\draw (-2,1,5) node[anchor=south] {.};\end{tikzpicture}}
\put(222,71){\begin{tikzpicture}\draw (-2,1,5) node[anchor=south] {.};\end{tikzpicture}}
\put(203,68){\begin{tikzpicture}\draw (-2,1,5) node[anchor=south] {.};\end{tikzpicture}}
\put(187,55){\begin{tikzpicture}\draw (-2,1,5) node[anchor=south] {.};\end{tikzpicture}}
\put(202,31){\begin{tikzpicture}\draw (-2,1,5) node[anchor=south] {.};\end{tikzpicture}}
\put(183,28){\begin{tikzpicture}\draw (-2,1,5) node[anchor=south] {.};\end{tikzpicture}}
\put(157,35){\begin{tikzpicture}\draw (-2,1,5) node[anchor=south] {.};\end{tikzpicture}}
\put(231,33){\begin{tikzpicture}\draw (-2,1,5) node[anchor=south] {.};\end{tikzpicture}}
\put(223,28){\begin{tikzpicture}\draw (-2,1,5) node[anchor=south] {.};\end{tikzpicture}}
\put(2017,35){\begin{tikzpicture}\draw (-2,1,5) node[anchor=south] {.};\end{tikzpicture}}

\put(144,66.5){\begin{tikzpicture}\draw (-2,1,5) node[anchor=south] {.};\end{tikzpicture}}
\put(160,61.5){\begin{tikzpicture}\draw (-2,1,5) node[anchor=south] {.};\end{tikzpicture}}
\put(176,58.5){\begin{tikzpicture}\draw (-2,1,5) node[anchor=south] {.};\end{tikzpicture}}
\put(192,57.4){\begin{tikzpicture}\draw (-2,1,5) node[anchor=south] {.};\end{tikzpicture}}
\put(208,58.){\begin{tikzpicture}\draw (-2,1,5) node[anchor=south] {.};\end{tikzpicture}}
\put(224,60.5){\begin{tikzpicture}\draw (-2,1,5) node[anchor=south] {.};\end{tikzpicture}}
\put(240,64.9){\begin{tikzpicture}\draw (-2,1,5) node[anchor=south] {.};\end{tikzpicture}}

\put(232,81){\begin{tikzpicture}\draw (-2,1,5)[opacity=0.4] node[anchor=south] {.};\end{tikzpicture}}
\put(218,101){\begin{tikzpicture}\draw (-2,1,5)[opacity=0.4] node[anchor=south] {.};\end{tikzpicture}}
\put(202,111){\begin{tikzpicture}\draw (-2,1,5)[opacity=0.4] node[anchor=south] {.};\end{tikzpicture}}
\put(172,101){\begin{tikzpicture}\draw (-2,1,5)[opacity=0.4] node[anchor=south] {.};\end{tikzpicture}}
\put(162,91){\begin{tikzpicture}\draw (-2,1,5)[opacity=0.4] node[anchor=south] {.};\end{tikzpicture}}
\put(182,101){\begin{tikzpicture}\draw (-2,1,5)[opacity=0.4] node[anchor=south] {.};\end{tikzpicture}}
\put(192,91){\begin{tikzpicture}\draw (-2,1,5)[opacity=0.4] node[anchor=south] {.};\end{tikzpicture}}
\put(197,125){\small$\bullet$}
\put(202,135){ $S$ CM}

\put(197,17){\small$\bullet$}
\put(202,5){ $\bar S$ }

\put(270,100){ $S'$ CM and $K_S^0=K_{S'}^0$}
\put(247,100){$\leftarrow$}
\put(280,70){$S'$ not CM}
\put(257,70){$\leftarrow$}
 \end{picture}

\subsection{Endomorphisms as Hodge classes on the product}\label{sec:EndoasHodge}
 Recall that a projective K3 surface $S$ has 
CM  if its transcendental part $T=T(S)$ has CM, i.e.\ the endomorphism field $K_S$ of
all endomorphisms of the Hodge structure $T$ is a CM field with $\dim_{K_S}T=1$.

\begin{remark}\label{rem:EndHodgeS2}
The endomorphism field $K_S={\rm End}_{\rm Hdg}(T(S))$  can be interpreted geometrically in terms of Hodge classes on the product $S\times S$.
More precisely,  via Poincar\'e duality the space of rational $(2,2)$ classes $H^{2,2}(S\times S,\QQ)$ splits as a direct sum of ${\rm End}({\rm NS}(S)_\QQ)\cong {\rm NS}(S)_\QQ\otimes {\rm NS}(S)_\QQ$
and $$K_S={\rm End}_{K_S}(T(S))\cong (T(S)\otimes T(S))^{2,2}.$$
Multiplication in $K_S$ can be understood in terms of composition of correspondences induced
by classes on $S\times S$.
As $T(S)\otimes T(S)\cong S^2T(S)\oplus \mbox{\Large $\wedge$}^2 T(S)$, one also has
$$K_S\cong (S^2T(S))^{2,2} \oplus (\mbox{\Large $\wedge$}^2 T(S))^{2,2}.$$ 
This is the decomposition into selfadjoint and anti-selfadjoint parts or, equivalently,
the eigenspace decomposition with respect to complex conjugation acting on $K_S$. In
particular,  the maximal totally real subfield is given by $$K_S^0\cong (S^2T(S))^{2,2}\subset H^{2,2}(S^{[2]},\QQ).$$
\end{remark}

As alluded to in the introduction, although $K_S^0=K_{S'}^0$ for all algebraic twistor fibres away from the equator, there is no class $\phi\in H^4(\ks\times_{\PP^1_\ell}\ks,\QQ)$ that would restrict to
a given class $\varphi\ne0,\pm1$ in $K_S^0$ on all the CM fibres. Indeed, $\varphi$ acts by multiplication
with a scalar on $\sigma$ and trivially on $\ell$. Thus, unless the scalar is $0$ or $\pm1$, it will not act
by scalar multiplication on any linear combination $a\sigma+b\bar\sigma+c\ell\ne\sigma,\bar\sigma$.

So, it seems something geometric happens behind the curtain of the transcendental twistor space, that is not completely explained by Hodge theory.

\subsection{K3 surfaces with CM are defined over number fields}\label{sec:CMK3}
The following important fact was proved in \cite[Thm.\ 4]{PSShaf} and a more precise version
can be found in \cite[Cor.\ 3.3.19]{Rizov}. The original arguments involve the Kuga--Satake construction
reducing the problem to the corresponding problem for abelian varieties. The proof below 
is more geometric relying on the Hodge conjecture for the square
$S\times S$ of a K3 surface $S$ with CM.

\begin{prop}\label{prop:CMimpliesQbar}
Any K3 surface with CM is defined over $\bar\QQ$. Equivalently, if $K_S=k_S$
for a projective K3 surface $S$, then $S$ is defined over $\bar\QQ$.
\end{prop}

\begin{proof} First note that by virtue of  Proposition \ref{prop:CMKk},  CM is indeed equivalent to the equality $K_S=k_S$.

Pick a primitive element $\alpha$ of $K_S$ of norm one, i.e.\ $\alpha$ generates $K_S$ and, viewed as an endomorphism of the Hodge structure $T(S)$, is an isometry. As it has been proved recently in \cite{Buskin},
$\alpha$ as a class in $H^{2,2}(S\times S,\QQ)$ is algebraic, see also \cite{HuyHC} for a motivic interpretation.

Now one applies the usual `spread out' technique to $S$ and  a cycle $Z$ on $S\times S$
representing $\alpha$. More precisely,
$S$ and $Z$ are both defined over a finitely generated field extension  $L/\bar\QQ$, i.e.\
$S\cong S_0\times_L\CC$ and $Z=Z_0\times_L\CC$. Viewing $L$ as
a function field of a variety $Y$ over $\bar\QQ$ and inverting denominators, we may consider $S_0$
as the scheme theoretic fibre $\ks_\eta$ of a family $\ks\to Y$. Taking the closure $\kz$
of $Z_0$ in $\ks\times_Y\ks$ and shrinking $Y$ if necessary produces a relative flat cycle. 
Next, let  $\ks_\CC\to Y_\CC$ be the base change to a family of complex K3 surfaces
and consider the action of the fibres of the cycle $\kz_\CC$ on the transcendental part $T(\ks_t)$,
$t\in Y(\CC)$. For very general $t\in Y(\CC)$,
the fibre $\ks_t$ as a scheme is isomorphic to the original $S$ via a base change with respect to a chosen
embedding $L\subset\CC$. The isomorphism is compatible with the $\ell$-adic action of $\kz_t$
and, therefore, the action of $[\kz_t]$ and $\alpha$ on the $\ell$-adic transcendental
part $T(\ks_t)_{\QQ_\ell}$ coincide. In particular, on that fibre the action of $\kz_t$
coincides with the action of $\alpha$ and, therefore, the classes
$[\kz_t]_*^k\gamma_t$, $k=1,\ldots$, span a subspace of dimension $[K_S:\QQ]$ (namely $T(\ks_t)$) for any $0\ne \gamma_t\in T(\ks_t)$. With $\gamma_t$ chosen to depend continuously on $t$, also the dimension of the subspace depends semi-continuously on  $t$ and, therefore, stays  constant on
nearby fibres $\ks_{t'}$
Hence, $\dim_\QQ T(\ks_{t'})\geq
\dim_\QQ T(S)$. However, the latter implies that $\rho(\ks_{t'})\leq \rho(\ks_t)=\rho(S)$ for all $t'$ in a neighbourhood of $t$,
which immediately implies that the family is isotrivial, cf.\ \cite[Prop.\ 6.2.9]{HuyK3} or \cite[Sec.\ 17.3.4]{Voisin}.
\end{proof}

\begin{remark}
A converse of Proposition \ref{prop:CMimpliesQbar}
was proved in \cite{Tretkoff}: If a projective K3 surface $S$ is defined over $\bar\QQ$
and its period field $k_S$ is algebraic, i.e.\ $k_S\subset\bar\QQ$, then $S$ has CM. 

This is a K3 analogue of the classical result for elliptic curves that the only elliptic curves
$\CC/(\ZZ\oplus\ZZ\cdot\tau)$ with $\tau$ and $j(\tau)$ algebraic are CM elliptic curves, i.e.\
when $\tau$ is imaginary quadratic.

Via the Kuga--Satake construction, the problem is reduced to the analogous statement for abelian varieties which had been settled by Cohen and Shiga--Wolfart \cite{Cohen,ShigaWolfart}, see also \cite[Thm.\ 1.3]{UllmoYafaev}. It would certainly
be interesting to find a geometric argument not relying on the Kuga--Satake construction.
\end{remark}

According to Corollary \ref{cor:notCM} we know that the algebraic fibres
$\ks_t$ over the equator do not have complex multiplication. This is strengthened by the following result.

\begin{cor}\label{cor:notQbar}
Assume $(S,L)$ is a polarized K3 surfaces with complex multiplication and Picard number $\rho(S)<20$. Then no algebraic fibre $\ks_t$ of the associated twistor space with $t$ contained in the equator is defined over $\bar\QQ$.
\end{cor}

\begin{proof}
Suppose $\ks_t$ were defined over $\bar\QQ$. According to Lemma \ref{lem:QbarQbar},
the period field $k_{\ks_t}$ is also algebraic. Then use \cite{Tretkoff} to conclude that
$\ks_t$ has complex multiplication which is excluded by Corollary \ref{cor:notCM}.
\end{proof}

\subsection{Period values of K3 surfaces with CM}\label{sec:Transc}

Assume that a K3 surface $S$ can be defined over $\bar\QQ$, i.e.\ there exists
a K3 surface $S^{\rm o}$ over $\bar \QQ$ such that $S\cong S^{\rm o}\times_{\bar \QQ} \CC$.
 Pick a regular form $0\ne\sigma^{\rm o}\in H^0(S^{\rm o},\Omega^2_{S^{\rm o}/\bar\QQ})$ 
 and consider $(\sigma^{\rm o}.\gamma)$ for a class  $0\ne \gamma\in T(S)$. These  values
 are all expected to be transcendental
 which would be the weight-two analogue of the classical fact that
 for an elliptic curve $E^{\rm o}$ over $\bar\QQ$ the integrals
 $\int_\delta\omega^{\rm o}$ with $\omega^{\rm o}\in H^0(E^{\rm o},\Omega_{E^{\rm o}/\bar\QQ})$ and
 $\delta\in H^1(E,\ZZ)$ are transcendental numbers. The problem for K3 surfaces seems open, but see \cite[Exa.\ 3]{Wuestholz}. Questions of irrationality have been dealt with in \cite{BostCharles}.

\begin{remark}
Often, in this context, the field $\QQ((\sigma^{\rm o}.\gamma))$ is called the period field, but it should not
be confused with the period field $k_S$ in the sense of Sections \ref{sec:ks} and \ref{sec:Dict}, which for example for K3 surfaces with CM is algebraic.
\end{remark}

Note that for a different choice of $\sigma^{\rm o}$  in $H^0(S^{\rm o},\Omega^2_{S^{\rm o}/\bar\QQ})$
the value of $(\sigma^{\rm o}.\gamma)$ changes by an algebraic factor. 
As a consequence of Lemma \ref{lem:PVindep}, one has the following
stronger fact.

\begin{cor}
Assume $S$ is a complex projective K3 surface with complex multiplication.
Then the \emph{period value} of $S$ 
$$r_{\scriptscriptstyle S}\coloneqq r_{\scriptscriptstyle T(S),\sigma^{\rm o}}\in\CC^\ast/\bar\QQ^\ast$$
is well-defined. Here, $S^{\rm o}$ is a model of $S$ over $\bar \QQ$ and $0\ne\sigma^{\rm o}\in H^0(S^{\rm o},\Omega^2_{S^{\rm o}/\bar\QQ})$.
\end{cor}

\begin{proof}
The model $S^{\rm o}$ is unique up to isomorphism and so is the induced
$\bar\QQ$-line
$H^0(S^{\rm o},\Omega^2_{S^{\rm o}/\bar\QQ})\subset H^0(S,\Omega^2_{S/\CC})$.
\end{proof}

Similarly, Lemma \ref{lem:r-1sigma} implies that
$$r_{\scriptscriptstyle S}^{-1}\cdot \sigma^{\rm o}\in T(S)\otimes\bar\QQ\text{ and }
(\sigma^{\rm o}.\bar\sigma^{\rm o})\equiv r_{\scriptscriptstyle S}\cdot \bar r_{\scriptscriptstyle S}\in \CC^\ast/\bar\QQ^\ast.$$

Currently, there are no techniques to compute or  guess the natural $\bar\QQ$-line
$H^0(S^{\rm o},\Omega^2_{S^{\rm o}/\bar\QQ})\subset H^0(S,\Omega^2_{S/\CC})$
for a K3 surface $S$ with CM. In the case $\rho(S)=20$, the problem reduces
to the computation of $\left(\int_\delta\omega^{\rm o}\right)^2$ for $\omega^{\rm o}\in H^0(E^{\rm o},\Omega_{E^{\rm o}/\bar\QQ})$. Here, $S$ is covered by the Kummer surface associated
with a product $E_1\times E_2$ of two elliptic curves both isogenous to a CM elliptic
curve $E^{\rm o}\times_{\bar \QQ}\CC$.

\begin{remark}
In principle one could hope that once the natural $\bar\QQ$-line for $S$ is found, the ones for all
other twistor fibres $\ks_T$ with CM can be predicted. However, at this point the
results of Section \ref{sec:PVTW} only tell us that
$r_{\scriptscriptstyle \ks_t}=a_t \cdot r_{\scriptscriptstyle S}$
assuming $\sigma_t^{\rm o}=a_t\sigma^{\rm o}+b_t\bar\sigma^{\rm o}+c_t\ell$, but how to
predict $a_t$, or equivalently $c_t$, seems unclear.

As according to Grothendieck's period conjecture any
 $\bar\QQ$-algebraic relation between $r_{\scriptscriptstyle S}$ and $r_{\scriptscriptstyle \ks_t}$ should
be motivic, we actually expect that infinitely many of the $r_{\scriptscriptstyle \ks_t}$ are independent of $r_{\scriptscriptstyle S}$, cf.\ \cite[Sec.\ 2.3]{HuyAnn}
\end{remark}


\end{document}